\documentclass[leqno,11pt]{amsart}
\usepackage{amssymb, amsmath,amsmath,latexsym,amssymb,amsfonts,amsbsy, amsthm,esint}
\usepackage{color}
\usepackage{amssymb,amsmath,amsthm,amsfonts,amsbsy,mathrsfs,dsfont}
\usepackage{paralist}
\usepackage[colorlinks=true]{hyperref}
\usepackage[pagewise]{lineno}%\linenumbers
\hypersetup{urlcolor=blue, citecolor=red}

\newtheorem{theorem}{Theorem}[section]

\newtheorem{lemma}[theorem]{Lemma}
\newtheorem{proposition}{Proposition}[section]

\theoremstyle{definition}
\newtheorem{definition}[theorem]{Definition}
\newtheorem{remark}{Remark}

\theoremstyle{remark}
\newtheorem{case}{Case}
\newtheorem{step}{Step}

\makeatletter % changes the catcode of @ to 11
\@addtoreset{case}{lemma}
\@addtoreset{case}{theorem}
\@addtoreset{case}{proposition}
\@addtoreset{case}{corollary}
\makeatother % changes the catcode of @ back to 12

\title[Fractional Camassa-Holm Equations] %Use the shortened version of the full title
      {On the Viscous Camassa-Holm Equations with Fractional Diffusion}

\author[Zaihui Gan, Fanghua Lin and Jiajun Tong]{}


 \email{ganzaihui2008cn@tju.edu.cn}
 \email{linf@cims.nyu.edu}
 \email{jiajun.tong@nyu.edu}

\thanks{$^*$ Corresponding author: Jiajun Tong}

\begin{document}
\maketitle

% Enter the first author's name and address:
\centerline{\scshape Zaihui Gan}
\medskip
{\footnotesize
% please put the address of the first author
 \centerline{Center for Applied Mathematics, Tianjin University}
   %\centerline{Other lines}
   \centerline{Tianjin 300072, China}
} % Do not forget to end the {\footnotesize by the sign }

\medskip

\centerline{\scshape Fanghua Lin and Jiajun Tong$^*$}
\medskip
{\footnotesize
 % please put the address of the second  and third author
 \centerline{Courant Institute, New York University}
   \centerline{251 Mercer Street}
   \centerline{New York, NY 10012, USA}
}

%\bigskip
%\centerline{\emph{Dedicated to Professor Wei-Ming Ni with Respect}}
%\bigskip
% The name of the associate editor will be entered by an editorial staff
% "Communicated by the associate editor name" is not needed for special issue.

%\centerline{(Communicated by the associate editor name)}

%The abstract of your paper
\begin{abstract}
We study Cauchy problem of a class of viscous Camassa-Holm equations (or Lagrangian averaged Navier-Stokes
equations) with fractional diffusion in both smooth bounded domains and in the whole space
in two and three dimensions. Order of the fractional diffusion is assumed to be $2s$
with $s\in [n/4,1)$, which seems to be sharp for the validity of the main results of
the paper; here $n=2,3$ is the dimension of space.
We prove global
well-posedness in $C_{[0,+\infty)}(D(A))\cap L^2_{[0,+\infty),loc}(D(A^{1+s/2}))$ whenever the initial data $u_0\in D(A)$, where $A$ is the Stokes operator.
%When $s = n/4$, global well-posedness is shown for \textcolor{red}{$\|u_0\|_{D(A)}$} being suitably small.
We also prove that such global solutions gain regularity instantaneously after the
initial time.
A bound on a higher-order spatial norm is also obtained.
\end{abstract}

\section{Introduction}\label{sec: introduction}

Hydrodynamic equations with nonlocal effects %or anomalous diffusion
have attracted a great attention in recent years.
While some of the problems are described by nonlocal equations to begin with, many
others, especially those concerning interface motion in fluids, are often derived from local equations.
See for examples \cite{caffarelli2010drift, constantin2017local,
ccg2011, ccf2005, constantin2016cgrs, DeLellis2017,
lin2019solvability} and references therein.
As an important type of nonlocality, fractional diffusion arises naturally in many hydrodynamic problems, characterizing %nonlocal feature of certain dynamics, such as
nonlocal drift or diffusion \cite{caffarelli2010drift,
constantin2017local, ccf2005, DeLellis2017}, or capturing certain thermal and electromagnetic effects \cite{constantin1994formation,constantin2017some}.
From an analytic point of view, evolution problems with these nonlocal features are of great interest on their own. %s from various analytic points of view.

In this paper, we shall study viscous Camassa-Holm equations with fractional diffusion
in $\Omega \subset\mathbb{R}^n$ $(n = 2,3)$.
Throughout the paper, unless otherwise stated, we shall always assume that
\begin{equation}
\Omega\subset\mathbb{R}^n \mbox{ is a smooth bounded domain, or } \Omega = \mathbb{R}^n,\mbox{ with }n = 2,3,
\label{eqn: assumption on the regularity of the domain}
\end{equation}
and
\begin{equation}
s\in[n/4,1).
\label{eqn: assumption on s}
\end{equation}
%, where
%$\Omega$ is a smooth bounded domain with boundary $\partial\Omega$, or $\Omega = \mathbb{R}^n$.
The equations are as follows:
\begin{align}
&\;\partial_t (1-\alpha^2\Delta)u + u\cdot \nabla (1-\alpha^2\Delta)u - \alpha^2\nabla u^T\cdot \Delta u +\nabla p  = -\nu (1-\alpha^2\Delta)A^s u,\label{eqn: LANS fractional equation}\\
&\; \mathrm{div}\,u = 0,\quad u|_{t = 0} = u_0\label{eqn: div free condition and initial data of LANS fractional}.
\end{align}
Here $u$ denotes a divergence-free fluid velocity field.
The constant $\alpha>0$ characterizes the scale at which fluid motion is averaged, and $\nu>0$ is the viscosity.
$A = \mathcal{P}(-\Delta)$ is the Stokes operator, with $\mathcal{P}$ being the Leray projection operator $\mathcal{P}:L^2(\Omega)\rightarrow \{v\in L^2(\Omega):\mathrm{div}\, v = 0,\,v\cdot n = 0\mbox{ on }\partial\Omega\}$; we always omit the $\Omega$-dependence of $A$ and $\mathcal{P}$.
$A^s$ with $s \in [\frac{n}{4},1)$ is the spectral fractional Stokes operator which will be defined
below, and which is a nonlocal operator in nature.
There are alternative (but not
necessarily equivalent) ways of defining fractional Stokes operators,
but we find the spectral fractional Stokes operator is the
easiest to work with for our purpose.
The range of $s$ is seen to be sharp from
the viewpoint of the energy method (see the proof of Theorem \ref{thm: global existence
and uniqueness} below).
Initial data for $u$ is specified.
When $\Omega$ is a smooth bounded domain, we additionally need
boundary conditions
\begin{equation}
u = A^s u = 0\quad \mbox{ on }\partial\Omega.\label{eqn: boundary conditions of LANS fractional}
\end{equation}

When $s= 1$, equations \eqref{eqn: LANS fractional equation}-\eqref{eqn: div free condition and
initial data of LANS fractional} are often referred as the classic viscous Camassa-Holm equations, or equivalently the isotropic Lagrangian averaged Navier-Stokes equations (LANS-$\alpha$) \cite{marsden2001global}. The inviscid version of the LANS-$\alpha$ equations, or the
Lagrangian averaged Euler (LAE-$\alpha$) equations, were first derived in
\cite{holm1998euler,holm1998euler_semidirect} from a variational formulation, motivated by the fact that the Camassa-Holm equation in one dimension describes geodesic motion on certain diffeomorphism group.
An alternative derivation can be found in \cite{holm1999fluctuation}.
Viscosities were later added to the LAE-$\alpha$ equations, giving rise to the LANS-$\alpha$
equations \cite{chen1998camassa,chen1999direct,chen1999camassa}.
Its relation to the turbulence theory has been well investigated \cite{foias2001navier,foias2002three,holm2002karman,mohseni2003numerical,cheskidov2004boundary}.
Both LAE-$\alpha$ and LANS-$\alpha$ equations can be viewed as closure models when motion at the scales smaller than $\alpha$ is averaged out.
Anisotropic generalizations of the LAE-$\alpha$ and LANS-$\alpha$ equations in bounded domains are presented in \cite{marsden2003anisotropic}, which takes into account that the covariance tensor of the Lagrangian fluctuation field is not constantly identity matrix throughout the domain and it should evolve with the flow.
For a more comprehensive history of the LANS-$\alpha$ equation, we refer the readers to \cite{marsden2001global} and the references therein.
As for results in analysis, a handful of global existence or well-posedness results of the
LANS-$\alpha$ equation have been established in periodic boxes \cite{foias2002three}, in bounded
domains and the whole space \cite{coutand2002global, bjorland2008questions, bjorland2008decay},
and on Riemannian manifolds with boundaries \cite{shkoller2000analysis}; decay of solutions
in bounded domains and the whole spaces was also investigated in \cite{bjorland2008questions,
bjorland2008decay}.

Although it is not obvious how fractional diffusion can be physically incorporated into derivations of
the Camassa-Holm equations, the specific form of the fractional dissipation in \eqref{eqn:
LANS fractional equation} together with the boundary conditions \eqref{eqn: boundary
conditions of LANS fractional} is quite natural from analysis point of view; a
similar choice is made in \cite{marsden2001global}.
For simplicity, we only focus on the isotropic fractional
LANS-$\alpha$ equations, i.e., the viscous Camassa-Holm equations, although it was
suggested that the anisotropic LANS-$\alpha$ equation may be more relavent for bounded
domains \cite{marsden2003anisotropic}.

Our first result, Theorem \ref{thm: global existence and uniqueness}, is the global well-posedness with sharp fractional power
$s$. It may be viewed as a fractional counterpart of classical results by Kieslev-Ladyzenskaya
and others for the Navier-Stokes equations \cite{kiselev1957existence,temam1984navier}.
It would also be interesting if one can build rather weak solutions as in
\cite{DeLellis2017} for suitable small positive powers $s$. Next, we show that the global
solution gains regularity when $t>0$, which is stated in Theorem \ref{thm: improved regularity in non-critical case} and
Theorem \ref{thm: improved regularity in critical case}. The latter, characterizing the critical case $(n,s) = (2,1/2)$, is in general not easy to
establish, and it may be a starting point for a further regularity theory. Here instead of
dealing with commutators associated with nonlocal operators on a bounded domain which could be rather technical, we make use of the fractional semigroups
to derive desired estimates.
One might need nonlocal commutator estimates when studying higher regularity and boundary regularity.
These related issues will be addressed elsewhere.

The rest of the paper is organized as follows. In Section \ref{section: preliminaries}, we
introduce the spectral fractional Stokes operator and present an equivalent formulation of
the equations \eqref{eqn: LANS fractional equation}-\eqref{eqn: div free condition and
initial data of LANS fractional}. Section \ref{section: global wellposedness} will be
devoted to proving Theorem \ref{thm: global existence and uniqueness} on
the global well-posedness result. %of the Camassa-Holm equations with fractional diffusion in two and three dimensions.
In Section \ref{section: improved regularity}, we prove that the
global solution enjoys higher spatial regularity for any positive time.
The main results are summarized in Theorem \ref{thm: improved regularity in non-critical case} for the non-critical case, and in
Theorem \ref{thm: improved regularity in critical case} for the critical case, respectively.
%which also includes a bound on
%a higher-order norm of the solution in space.

\section{Preliminaries}\label{section: preliminaries}
%Before defining the operator $A^s$,
We first introduce some notations.
Let $\Sigma =\{\phi \in C_0^\infty(\Omega):\nabla\cdot \phi=0\}$.
As in much literature on mathematical hydrodynamics,
let $V$ denote the $H^1$-completion of $\Sigma$; while the $L^2$-completion of $\Sigma$ is denoted by $H$. %will be denoted by $H^{m}_{0}(\Omega)$ and $(H^{m}_{0}(\Omega))'$ be the dual space.
Define $V^r = H^r(\Omega)\cap V$ for all $r\geq 1$; obviously $V = V^1$.

In the case of $\Omega = \mathbb{R}^n$, $(-\Delta)$ and $\mathcal{P}$ are both Fourier multipliers, and thus they commute.
% when applied to $V^r$-functions.
Indeed, for all $f \in\mathscr{S}(\mathbb{R}^n)$, $\widehat{Af}(\xi) = \hat{\mathcal{P}}(\xi)|\xi|^2 \hat{f}(\xi)$, where
\begin{equation*}
\hat{f}(\xi) = \frac{1}{(2\pi)^{n/2}}\int_{\mathbb{R}^n} f(x)e^{-i x\cdot \xi}\,dx
\end{equation*}
is the Fourier transform of $f$, and where $\hat{\mathcal{P}}(\xi)$ is the Fourier multiplier associated with $\mathcal{P}$.
Hence, $A^s$ can be naturally defined by
\begin{equation*}
\widehat{A^s f}(\xi) = \hat{\mathcal{P}}(\xi)|\xi|^{2s} \hat{f}(\xi).
\end{equation*}
Define
\begin{equation*}
\|f\|_{D(A^r)(\mathbb{R}^n)} := (\|A^r f\|_{L^2(\mathbb{R}^n)}^2+\|f\|_{L^2(\mathbb{R}^n)}^2\mathds{1}_{\{r>0\}})^{1/2}.
\end{equation*}
%In this case, the equation \eqref{eqn: LANS fractional equation}-\eqref{eqn: div free condition and initial data of LANS fractional} can be rewritten as
%\begin{align*}
%&\;\partial_t v+ u\cdot \nabla v + v\cdot \nabla u^T +\nabla p  = -\nu (-\Delta)^s v,\\%\label{eqn: C-H fractional equation}\\
%&\;v = (1-\alpha^2\Delta)u,\\%\label{eqn: relation between u and v in C-H equation}\\
%&\;\mathrm{div}\,u = 0,\quad u|_{t = 0} = u_0.\\%\label{eqn: div free condition and initial data of C-H fractional}.
%\end{align*}
%This simpler form coincides with the Camassa-Holm equation in $\mathbb{R}^n$ \cite{bjorland2008questions} with fractional viscosity.

Now consider $\Omega\subset \mathbb{R}^{n}$ $(n=2,3)$ to be a smooth bounded domain.
For the stationary Stokes equation in $\Omega$ with zero Dirichlet boundary condition, there exists a sequence of eigenvalues $\{\mu_j\}_{j\in\mathbb{Z}_+}\subset \mathbb{R}_+$ and a sequence of eigenfunctions $\{w_j\}_{j\in \mathbb{Z}_+}\subset L^2(\Omega)$, both depending on $\Omega$, solving
\begin{equation*}
A w_j=\mu_j w_j\mbox{ in }\Omega,\quad \mathrm{div}\, w_j = 0,\quad w_j|_{\partial\Omega}=0,
\end{equation*}
such that $\{\mu_j\}_{j\in\mathbb{Z}_+}$ is non-decreasing in $\mathbb{R}_+$ and  $\{w_j\}_{j\in \mathbb{Z}_+}$ forms an orthonormal basis of $H$. %, where $L^2_\sigma(\Omega)$ is the $L^2$-completion of the set of $C^\infty_0(\Omega)$-divergence-free vector fields.
%In particular, $\int_{\Omega} w_j w_k \,dx = \delta_{jk}$.
%Here $p_j$ is determined a posteriori due to divergence-free condition on $w_j$.
It is known that $w_j\in C^\infty(\Omega)\cap V$ \cite{temam1984navier}.
%Let $H_m=span\{w_1,\cdots,w_m\}$, and $\mathbb{P}_m:L^2_\sigma(\Omega)\rightarrow H_m$ be the $L^2$-orthogonal projection from $L^2_\sigma$ onto $H_m$.
For all $f\in H$, it has a spectral decomposition
\begin{equation*}
f(x)=\sum_{j=1}^\infty f_j w_j(x),\quad f_j=\int_{\Omega}f(x) w_j(x)\,dx.
\end{equation*}
The infinite sum is understood in the $L^2$-sense.
In fact, $\|f\|_{L^2(\Omega)} = \|\{f_j\}_{j\in\mathbb{Z}_+}\|_{l^2}$.
For all $r\in\mathbb{R}$, define
\begin{equation*}
D(A^r)(\Omega) = \left\{f(x)=\sum_{j=1}^\infty f_j w_j(x): \{\mu_j^rf_j\}_{j\in\mathbb{Z}_+}\in l^2\right\},
\end{equation*}
with
\begin{equation*}
\|f\|_{D(A^r)(\Omega)} := (\|\{\mu_j^r f_j\}_{j\in\mathbb{Z}_+}\|_{l^2}^2+\|\{f_j\}_{j\in \mathbb{Z}_+}\|_{l^2}^2\mathds{1}_{\{r>0\}})^{1/2}.
\end{equation*}
We shall omit the $\Omega$-dependence in $D(A^r)(\Omega)$ whenever it is convenient.
Then for all $f\in D(A^r)$,
\begin{equation*}
A^r f(x) := \sum_{j=1}^\infty \mu_j^r f_j w_j(x).
\end{equation*}
Again the infinite sum is understood in the $L^2$-sense.
As a result, $\|f\|_{D(A^r)(\Omega)} = (\|A^r f\|_{L^2(\Omega)}^2+\|f\|_{L^2(\Omega)}^2\mathds{1}_{\{r>0\}})^{1/2}$.

Note that when $\Omega$ is a smooth bounded domain, the boundary condition \eqref{eqn: boundary conditions of LANS fractional} is well-defined and it is automatically satisfied in the space $D(A^{1+s/2})$.
Indeed, we have the following lemma.
\begin{lemma}\label{lemma: boundary data of D(A^r) functions}
Suppose $\Omega\subset\mathbb{R}^n$ $(n = 2,3)$ is a smooth bounded domain. Then all $D(A^r)$-functions have trace zero if $r> 1/4$.
\begin{proof}
We first show that $D(A)=V^2$.
Indeed, for all $f\in D(A)$, there is a sequence of $\{f^n\}\subset C^\infty(\Omega)\cap V$ being finite linear combinations of $w_j$, such that $f^n\to f$ in the $D(A)$-norm.
On one hand, this implies that $\{Af^n\}_{n\in\mathbb{Z}_+}$ forms a Cauchy sequence in $L^2(\Omega)$ and thus $\{f^n\}_{n\in\mathbb{Z}_+}$ is a Cauchy sequence in $V^2$ thanks to their zero boundary conditions and the regularity theory of the stationary Stokes equation %ellipticity of $A$
\cite{temam1984navier,lions1972non}.
We assume $f_n\rightarrow f_*$ in $V^2$ for some $f_*\in V^2$.
On the other hand, that $f^n \to f$ in $D(A)$ implies that $f^n\rightarrow f$ in $L^2(\Omega)$.
Hence, $f = f_*\in V^2$, and $\|f\|_{H^2(\Omega)}\leq C\|Af\|_{L^2(\Omega)} \leq C\|f\|_{D(A)}$ because the same estimates holds for $f^n$.
This implies that $D(A)\hookrightarrow V^2$.
That $V^2 \hookrightarrow D(A)$ is trivial since $\|Af\|_{L^2(\Omega)}\leq C\|f\|_{H^2(\Omega)}$.

To this end, for all $f\in D(A)$, we have $f\in L^2$  with $f|_{\partial\Omega} = 0$ and $(-\Delta)f\in L^2$.
This gives
\begin{equation*}
\|f\|_{D(A^{1/2})}^2 = \sum_{j=1}^\infty (\mu_j+1) f_j^2 = \int_\Omega f(1-\Delta)f \,dx = \|f\|_{H^1(\Omega)}^2.
\end{equation*}
This is also true for all $f\in D(A^{1/2})$ since $D(A)$ is dense in $D(A^{1/2})$.
Hence, $D(A^{1/2})=V$.

Since the embeddings $i:D(A^0)\rightarrow L^2(\Omega)$ and $i: D(A^{1/2})\rightarrow H_0^1(\Omega)$ are continuous, by interpolation, for $r\in(1/4,1/2]$, $i$ is continuous from $[D(A^0), D(A^{1/2})]_{2r} = D(A^r)$ to $[L^2(\Omega), H_0^1(\Omega)]_{2r}= H_0^{2r}(\Omega)$ \cite{lions1972non}.
Functions in $H_0^{2r}(\Omega)$ all have zero trace.
For $r>1/2$, it suffices to note that $D(A^r)\hookrightarrow D(A^{1/2})$.
\end{proof}
\begin{remark}\label{rmk: relation between D(A^r) and V^2r}
When $\Omega$ is a smooth bounded domain, in fact, $D(A^r) = V^{2r}$ for all $r\in[1/2,5/4)$.
More generally, by virtue of the interpolation theory \cite{lions1972non}, $D(A^r)\hookrightarrow V^{2r}$ for all $r\geq 1/2$.
When $\Omega=\mathbb{R}^n$, $D(A^r)=V^{2r}$ for all $r\geq 1/2$.
\end{remark}
\end{lemma}

Let $(1-\alpha^2\Delta)^{-1}$ be the inverse of the elliptic operator $(1-\alpha^2\Delta)$ on $\Omega$ (with zero Dirichlet boundary condition if $\Omega$ is a smooth bounded domain).
In the view of %$u\in D(A^{1+s/2})$ and Lemma \ref{lemma: boundary data of D(A^r) functions},
$A^s u = 0$ on $\partial\Omega$ if $\Omega$ is a smooth bounded domain,
%Then
it is valid to take $(1-\alpha^2\Delta)^{-1}$ on both sides of \eqref{eqn: LANS fractional equation}, and we obtain the following equivalent formulation of the Camassa-Holm equation with fractional diffusion \cite{marsden2001global}:
%\begin{enumerate}
%\item
%\begin{equation}
%\partial_t u +u\cdot \nabla u +\mathcal{U}^\alpha(u) +(1-\alpha^2\Delta)^{-1}\nabla p = -\nu A^s u,
%\end{equation}
%where
%\begin{equation}
%\mathcal{U}^\alpha(u) = \alpha^2 (1-\alpha^2 \Delta)^{-1} \mathrm{div}\,[\nabla u\cdot \nabla u^T + \nabla u\cdot \nabla u - \nabla u^T\cdot\nabla u],
%\end{equation}
%and $(1-\alpha^2\Delta)^{-1}$ is the inverse of the elliptic operator $(1-\alpha^2\Delta)$ on $\Omega$ (with $0$ Dirichlet boundary condition in the case of $\Omega$ being a smooth bounded domain)
%\item
\begin{equation}
\partial_t u +\nu A^s u + \mathcal{P}^\alpha [u\cdot\nabla u +\mathcal{U}^\alpha(u,u)] = 0,
\label{eqn: equivalent formulation of LANS equation}
\end{equation}
where with adaptation of notations in \cite{marsden2001global},
\begin{equation}
\mathcal{U}^\alpha(u_1,u_2) = \alpha^2 (1-\alpha^2 \Delta)^{-1} \mathrm{div}\,[\nabla u_1\cdot \nabla u_2^T + \nabla u_1\cdot \nabla u_2 - \nabla u_1^T\cdot\nabla u_2],
\label{eqn: def of U}
\end{equation}
and $\mathcal{P}^\alpha: H_0^1\cap H^r(\Omega)\rightarrow V^r$ $(r\geq 1)$ is the \emph{Stokes projector} \cite{shkoller2000analysis} uniquely defined by
\begin{equation*}
\begin{split}
&\;(1-\alpha^2 \Delta)\mathcal{P}^\alpha(w) +\nabla p = (1-\alpha^2 \Delta)w,\\
&\;\mathrm{div}\,\mathcal{P}^\alpha(w) = 0,\quad\mathcal{P}^\alpha(w)|_{\partial\Omega} = 0.
\end{split}
\end{equation*}
For all $r\geq 1$, $\mathcal{P}^\alpha$ is bounded from $H_0^1\cap H^r(\Omega)$ to $V^r$ \cite{marsden2001global}.
%\eqref{eqn: equivalent formulation of LANS equation} will be the starting point of our analysis.

The relation between \eqref{eqn: LANS fractional equation}-\eqref{eqn: div free condition and initial data of LANS fractional} and the viscous Camassa-Holm equation can be formally revealed as follows.
Assuming sufficient regularity of $u$, we apply $\mathcal{P}$ to \eqref{eqn: LANS fractional equation} to obtain that
\begin{equation}
\partial_t (1+\alpha^2 A)u + \mathcal{P}[u\cdot \nabla (1-\alpha^2\Delta)u - \alpha^2\nabla u^T\cdot \Delta u ] = -\nu (1+\alpha^2 A)A^s u.
\label{eqn: projected equation for u}
\end{equation}
Suppose $-\Delta u = Au + \nabla q$ for some $q$ and define
\begin{equation*}
v = (1+\alpha^2 A)u.
\end{equation*}
Then \eqref{eqn: projected equation for u} becomes
\begin{equation*}
\partial_t v + \mathcal{P}[u\cdot \nabla v+\alpha^2 u\cdot \nabla \nabla q - \alpha^2\nabla u^T\cdot \Delta u ] = -\nu A^s v.
\end{equation*}
Since
\begin{equation*}
\begin{split}
\alpha^2\mathcal{P}[ u\cdot \nabla \nabla q - \nabla u^T\cdot \Delta u ] = &\; \alpha^2\mathcal{P}[ -\nabla u^T\cdot \nabla q - \nabla u^T\cdot \Delta u ]\\
 =&\; \alpha^2\mathcal{P}[ \nabla u^T\cdot Au] = \mathcal{P}[ \nabla u^T v],
\end{split}
\end{equation*}
we obtain that
\begin{equation}
\partial_t v + u\cdot \nabla v+\nabla u^T  v +\nabla \tilde{p}= -\nu A^s v
\label{eqn: v equation}
\end{equation}
for some $\tilde{p}$.
This recovers the more commonly-used form of the Camassa-Holm equation \cite{bjorland2008questions} with fractional diffusions.

\section{Global Well-posedness}\label{section: global wellposedness}
Our main result on the global well-posedness of the equations \eqref{eqn: equivalent formulation of LANS equation} and \eqref{eqn: div free condition and initial data of LANS fractional} (or equivalently, \eqref{eqn: LANS fractional equation}-\eqref{eqn: div free condition and initial data of LANS fractional}), with boundary conditions \eqref{eqn: boundary conditions of LANS fractional} when $\Omega$ is a smooth bounded domain, is as follows.
\begin{theorem}[Global well-posedness]\label{thm: global existence and uniqueness}
Assume \eqref{eqn: assumption on the regularity of the domain} and \eqref{eqn: assumption on s}, and let $u_0 \in D(A)$.
Then there exists a unique solution $u\in C_{[0,+\infty)}(D(A))\cap L^2_{[0,+\infty),loc}(D(A^{1+s/2}))$ with $\partial_t u\in L_{[0,+\infty),loc}^2 (D(A^{1-s/2}))$ solving \eqref{eqn: equivalent formulation of LANS equation} with initial condition $u|_{t = 0} = u_0$ (and boundary conditions \eqref{eqn: boundary conditions of LANS fractional} if $\Omega$ is a smooth bounded domain).
    It satisfies
    \begin{equation}
    \|u\|_{L^\infty_{[0,+\infty)} (D(A))}+ \|A^{1+s/2}u\|_{L^2_{[0,+\infty)}L^2}\\
\leq C\|u_0\|_{D(A)}.
    \label{eqn: estimate of the global solution}
    \end{equation}
    where $C = C(\alpha, s,n, \nu, \Omega, \|u_0\|_{D(A^{1/2})})$.
%\item If $s = n/4$, there exists an $\varepsilon_* = \varepsilon_*(\alpha,\nu,s,n,\Omega)>0$, such that if $\|u_0\|_{D(A)}\leq \varepsilon_*$, there exists a unique solution $u\in C_{[0,+\infty)}(D(A))\cap L^2_{[0,+\infty)}(D(A^{1+s/2}))$, solving \eqref{eqn: equivalent formulation of LANS equation} with initial condition $u|_{t = 0} = u_0$.
%    For all $T_*>0$, it satisfies
%    \begin{equation}
%    \|u\|_{L^\infty_{T_*}(D(A))}+\|A^{1+s/2}u\|_{L^2_{T_*}L^2}\leq C \|u_0\|_{D(A)},
%    \label{eqn: estimate of the global solution critical case}
%    \end{equation}
%    where $C = C(\nu)<+\infty$ is a universal constant.
%\end{enumerate}
\end{theorem}

%The following lemma is a refined version of the ordinary interpolation inequality, which will be useful in the rest of the paper.
%\begin{lemma}\label{lemma: special interpolation inequality}
%For all $u\in D(A^{1+s/2})$ with $s\in (0,1)$, and all $r\in[0,s]$,
%\begin{equation}
%\|\nabla u\|_{H^{1+r}(\Omega)}\leq C\|u\|_{D(A)}^{1-r/s}\|A^{s/2}u\|_{D(A)}^{r/s},
%\end{equation}
%where $C = C(r,n,\Omega)$.
%\begin{proof}
%Let $\theta = r/s$.
%\setcounter{case}{0}
%\begin{case}[$\Omega = \mathbb{R}^n$]
%By Cauchy-Schwarz inequality,
%\begin{equation}
%\begin{split}
%\|\nabla u\|_{H^{1+r}(\Omega)}^2 \leq &\;C\int_{\mathbb{R}^n}(|\xi|^2+|\xi|^{2(2+r)})|\hat{u}(\xi)|^2\,d\xi \\
%\leq &\;C\int_{\mathbb{R}^n}(1+|\xi|^2)^{2(1-\theta)}(|\xi|^s(1+|\xi|^2))^{2\theta}|\hat{u}(\xi)|^2\,d\xi\\
%\leq &\;C\left(\int_{\mathbb{R}^n}(1+|\xi|^2)^2|\hat{u}(\xi)|^2\,d\xi\right)^{1-\theta}\left(\int_{\mathbb{R}^n}|\xi|^{2s}(1+|\xi|^2)^2|\hat{u}(\xi)|^2\,d\xi\right)^{\theta}\\
%\leq &\;C\|u\|_{D(A)}^{2(1-\theta)}\|A^{s/2}u\|_{D(A)}^{2\theta}.
%\end{split}
%\end{equation}
%\end{case}
%\begin{case}[$\Omega\subset \mathbb{R}^n$ is a smooth bounded domain]
%Since $D(A^{1+r/2})\hookrightarrow H^{2+r}(\Omega)$,
%\begin{equation}
%\|\nabla u\|_{H^{1+r}(\Omega)}\leq C\|u\|_{D(A^{1+r/2})} \leq C\|u\|_{D(A)}^{1-\theta} \|u\|_{D(A^{1+s/2})}^\theta\leq C\|u\|_{D(A)}^{1-\theta} \|A^{s/2}u\|_{D(A)}^\theta.
%\end{equation}
%Here we used the fact that $\|u\|_{D(A^{1+s/2})}\leq C\|A^{s/2}u\|_{D(A)}$ when $\Omega$ is a smooth bounded domain.
%\end{case}
%\end{proof}
%\end{lemma}

As the first step towards the global well-posedness, the following proposition states the local well-posedness result.
Note that in the 3-D case, the range of $s$ for the local well-posedness is wider than that in the global well-posedness result.
\begin{proposition}[Local well-posedness]\label{prop: local existence and uniqueness}
Assume \eqref{eqn: assumption on the regularity of the domain} and $s\in[\frac{1}{2},1)$, and let $u_0 \in D(A)$.
Then there exists $T = T(\alpha,\nu,s,n,\Omega, u_0)>0$ and a unique solution $u\in C_{[0,T]}(D(A))\cap L^2_T(D(A^{1+s/2}))$ with $\partial_t u\in L_{T}^2 (D(A^{1-s/2}))$ solving \eqref{eqn: equivalent formulation of LANS equation} with initial condition $u|_{t = 0} = u_0$ (and boundary conditions \eqref{eqn: boundary conditions of LANS fractional} if $\Omega$ is a smooth bounded domain), which satisfies
\begin{equation}
\|u\|_{L^\infty_T (D(A))}^2+ \nu\int_0^T \|A^{1+s/2}u\|_{L^2}^2\,dt\\
\leq C\|u_0\|_{D(A)}^2.
\label{eqn: estimate of the local solution}
\end{equation}
where $C$ is a universal constant.
%\item If $(n,s) = (2,1/2)$, there exists an $\varepsilon = \varepsilon(\alpha,\nu,\Omega)>0$, such that if $\|u_0\|_{D(A)}\leq \varepsilon$, there exists a unique solution $u(x,t)\in C_{[0,T]}(D(A))\cap L^2_T(D(A^{1+s/2}))$ with $T = 1$, solving \eqref{eqn: equivalent formulation of LANS equation} with initial condition $u|_{t = 0} = u_0$.
%    It satisfies \eqref{eqn: estimate of the local solution}.
%\end{enumerate}

%\begin{enumerate}
%\item If \eqref{eqn: assumption on s} holds with $(n,s)\not = (2,1/2)$, there exists a $T = T(\alpha,\nu,s,n,\Omega, \|u_0\|_{D(A)})>0$, such that there exists a unique solution $u(x,t)\in C_{[0,T]}(D(A))\cap L^2_T(D(A^{1+s/2}))$ solving \eqref{eqn: equivalent formulation of LANS equation} with initial condition $u|_{t = 0} = u_0$.
%    It satisfies
%    \begin{equation}
%    \|u\|_{L^\infty_T(D(A))\cap L^2_T(D(A^{1+s/2}))} \leq C \|u_0\|_{D(A)},
%    \label{eqn: estimate of the local solution}
%    \end{equation}
%    where $C = C(\nu)$.
%\item If $(n,s) = (2,1/2)$, there exists an $\varepsilon = \varepsilon(\alpha,\nu,\Omega)>0$, such that if $\|u_0\|_{D(A)}\leq \varepsilon$, there exists a unique solution $u(x,t)\in C_{[0,T]}(D(A))\cap L^2_T(D(A^{1+s/2}))$ with $T = 1$, solving \eqref{eqn: equivalent formulation of LANS equation} with initial condition $u|_{t = 0} = u_0$.
%    It satisfies \eqref{eqn: estimate of the local solution}.
%\end{enumerate}

\begin{proof}
%Let $M = \|u_0\|_{D(A)}<+\infty$.
%%If $M = 0$, the lemma is trivial; we only consider the case $M>0$.
%With $T\in(0,1]$ to be determined, we define
%\begin{equation*}
%\begin{split}
%B \triangleq \{u\in C_{[0,T]}(D(A))&\cap L^2_T(D(A^{1+s/2})): u|_{t = 0} = u_0,\\
%&\;\left.\|u\|^2_{L^\infty_T(D(A))}+\nu\|u\|^2_{L^2_T (D(A^{1+s/2}))}\leq 4M^2\right\}.
%\end{split}
%%\label{eqn: definition of B in considering the fixed point problem}
%\end{equation*}
%With $T\ll 1$, $B$ is non-empty (as $e^{-t\nu A^s}u_0\in B$) and closed in $C_{[0,T]}(D(A))\cap L^2_T(D(A^{1+s/2}))$.
%%It is clear that For all $u\in B$, the estimate \eqref{eqn: estimate of the local solution} holds because of the definition of $B$ and the energy estimate for the homogeneous solution $e^{-t\nu A^s}u_0$.
%%
%%Consider the map $Q:B\rightarrow C_{[0,T]}(D(A))\cap L^2_T(D(A^{1+s/2}))$, with $w = Qu$, where
%Taking $u\in B$, let $w\in C_{[0,T]}(D(A))\cap L^2_T(D(A^{1+s/2}))$ be the unique weak solution of
%\begin{equation}
%\partial_t w + \nu A^s w +\mathcal{P}^\alpha[u\cdot \nabla w]= -\mathcal{P}^\alpha \mathcal{U}^\alpha(u,u),\quad w|_{t = 0} = u_0,
%\label{eqn: fixed point problem}
%\end{equation}
%in the sense that for any $v\in $
%In order to show existence and uniqueness of such $w$ in the desired space,

The main ingredient of the proof is the Galerkin approximation.
We proceed in two different cases.

\setcounter{case}{0}
\begin{case}
In this case, assume $\Omega$ is a smooth bounded domain.
\setcounter{step}{0}
\begin{step}
For any $r\geq 0$, let $\mathcal{P}_N$ be the orthogonal projection from $V^r$ to $V_N = \mathrm{span}\{w_1,\cdots, w_N\}$, where $w_j$'s are eigenfunctions of the Stokes operator defined in Section \ref{section: preliminaries}.
Let $u^N $ solve
\begin{equation}
\partial_t u^N + \nu A^s u^N +\mathcal{P}_N\mathcal{P}^\alpha [u^N\cdot \nabla u^N+\mathcal{U}^\alpha(u^N,u^N)]=0,\quad u^N|_{t = 0} = \mathcal{P}_N u_0.
\label{eqn: fixed point problem in finite dim space}
\end{equation}
To construct such $u^N$, assume $u^N = \sum_{j=1}^N a_j(t)w_j$.
Then \eqref{eqn: fixed point problem in finite dim space} can be written as
an ODE system for $a_j$'s, i.e.,
\begin{equation*}
\frac{da_j}{dt} + \nu \mu_j^s a_j + \sum_{k,l = 1}^N a_k a_l[\langle w_j,w_k\cdot\nabla w_l \rangle +\langle w_j,\mathcal{U}^\alpha(w_k,w_l) \rangle],\quad a_j(0) = \langle w_j, u_0\rangle.
\end{equation*}
Here $\langle\cdot,\cdot\rangle$ denotes the $L^2$-inner product on $\Omega$.
It is not hard to show the inner products all have finite values. %$\mathcal{U}^\alpha(w_k,w_l)\in H^1(\Omega)$ (see a brief derivation below).
Then local existence and uniqueness of $u^N$ follows from the classic ODE theory.

\end{step}

\begin{step}
We shall derive energy estimates for $u^N$.
It is straightforward to find that
\begin{equation}
\frac{1}{2}\frac{d}{dt}\|u^N\|_{L^2}^2 + \nu \|A^{s/2}u^N\|_{L^2}^2 = -\langle u^N,\mathcal{U}^\alpha(u^N,u^N)\rangle.
\label{eqn: low order energy estimate crude form}
\end{equation}
By the definition of $\mathcal{U}^\alpha$ in \eqref{eqn: def of U}, for arbitrary $v_1,v_2\in D(A)$,
\begin{equation*}
\begin{split}
&\;\|\mathcal{U}^\alpha(v_1,v_2)\|_{H_0^1(\Omega)}\\
\leq &\;C\|\nabla v_1\cdot \nabla v_2^T+\nabla v_1\cdot \nabla v_2-\nabla v_1^T\cdot \nabla v_2\|_{L^2}\\
\leq &\;C\|v_1\|_{D(A)}\|v_2\|_{D(A)},
\end{split}
%\label{eqn: H^1 estimate of source term}
\end{equation*}
which yields
\begin{equation}
\frac{1}{2}\frac{d}{dt}\|u^N\|_{L^2}^2 + \nu \|A^{s/2}u^N\|_{L^2}^2 \leq C\|u^N\|_{D(A)}^2 \|u^N\|_{L^2}.
\label{eqn: low order energy estimate}
\end{equation}
%Thanks to \eqref{eqn: estimate of the local solution}, $w^N$ enjoys a uniform-in-$N$ bound in $L^\infty_T L^2\cap L^2_T D(A^{s/2})$.
%Similarly, f
%
%
Next, we derive a higher order estimate.
Taking inner product of \eqref{eqn: fixed point problem in finite dim space} and $A^2 u^N$,
\begin{equation}
\frac{1}{2}\frac{d}{dt}\|A u^N\|_{L^2}^2 + \nu \|A^{1+s/2}u^N\|_{L^2}^2 = -\langle A^2 u^N,u^N\cdot \nabla u^N\rangle -\langle A^2 u^N,\mathcal{U}^\alpha(u^N,u^N)\rangle.
\label{eqn: higher energy estimate crude form}
\end{equation}
Since $u^N$ is smooth, assuming $-\Delta u^N = Au^N+\nabla p^N$ for some $p^N$, we derive that
\begin{equation}
\begin{split}
&\;\langle A^2 u^N,u^N\cdot \nabla u^N\rangle\\
=&\;\langle A u^N,(-\Delta)(u^N\cdot \nabla u^N)\rangle\\
=&\;\langle A u^N,u^N\cdot \nabla (-\Delta)u^N -\Delta u^N\cdot \nabla u^N - 2\partial_k u_j^N\partial_{jk}u^N \rangle\\
=&\;\langle A u^N,u^N\cdot \nabla \nabla p^N -\Delta u^N\cdot \nabla u^N - 2\partial_k u_j^N\partial_{jk}u^N \rangle\\
=&\;\langle A u^N,-(\nabla u^N)^T\cdot \nabla p^N -\Delta u^N\cdot \nabla u^N - 2\partial_k u_j^T\partial_{jk}u^N \rangle\\
=&\;\langle A u^N,(\nabla u^N)^T (Au^N+\Delta u^N) -\Delta u^N\cdot \nabla u^N - 2\partial_k u_j^N\partial_{jk}u^N \rangle.
\end{split}
\label{eqn: explore cancellation in the transport term}
\end{equation}
Combining this with Remark \ref{rmk: relation between D(A^r) and V^2r} and the assumption $s\in [ \frac{1}{2},1)$,
\begin{equation}
\begin{split}
&\;|\langle A^2 u^N,u^N\cdot \nabla u^N\rangle|\\
\leq &\;C\| A u^N\|_{L^{\frac{2n}{n-2s}}}\|\nabla u^N\|_{L^{\frac{n}{s}}}(\|Au^N\|_{L^2}+\|\nabla^2 u^N\|_{L^2})\\
\leq &\;C\|u^N\|_{D(A^{1+s/2})}\|u^N\|_{D(A^{(1+\frac{n}{2}-s)/2})}\|u^N\|_{D(A)}\\
\leq &\;C\|u^N\|_{D(A^{1+s/2})}\|u^N\|_{D(A)}^2.
\end{split}
\label{eqn: higher order estimate for inner product of the transport term}
\end{equation}
In addition, for arbitrary $v_1,v_2\in D(A^{1+s/2})$,
\begin{equation*}
\|\mathcal{P}^\alpha\mathcal{U}^\alpha(v_1,v_2)\|_{D(A)}
\leq C\|\nabla v_1\cdot (\nabla v_2)^T+\nabla v_1\cdot \nabla v_2-(\nabla v_1)^T\cdot \nabla v_2\|_{H^1}.
\end{equation*}
%For $n = 3$, b
By Sobolev embedding and Remark \ref{rmk: relation between D(A^r) and V^2r},
\begin{equation}
\begin{split}
%&\;\|\mathcal{P}^\alpha\mathcal{U}^\alpha(v_1,v_2)\|_{D(A)}\\
%\leq
&\;\|\nabla v_1\cdot (\nabla v_2)^T+\nabla v_1\cdot \nabla v_2-(\nabla v_1)^T\cdot \nabla v_2\|_{H^1}\\
\leq &\;C\|\nabla^2 v_1\|_{L^{\frac{2n}{n-2s}}}\|\nabla v_2\|_{L^{\frac{n}{s}}}+C\|\nabla v_1\|_{L^{\frac{n}{s}}}\|\nabla^2 v_2\|_{L^{\frac{2n}{n-2s}}}\\
\leq &\;C\|v_1\|_{D(A^{1+s/2})}\|v_2\|_{D(A^{(1+\frac{n}{2}-s)/2})}+C\| v_1\|_{D(A^{(1+\frac{n}{2}-s)/2})}\|v_2\|_{D(A^{1+s/2})},
\end{split}
\label{eqn: D(A) estimate of source term v1 v2}
\end{equation}
%Likewise, for $n = 2$,
%\begin{equation}
%\begin{split}
%%&\;\|\mathcal{P}^\alpha\mathcal{U}^\alpha(v_1,v_2)\|_{D(A)}\\
%%\leq
%&\;\|\nabla v_1\cdot (\nabla v_2)^T+\nabla v_1\cdot \nabla v_2-(\nabla v_1)^T\cdot \nabla v_2\|_{H^1}\\
%\leq &\;C\|\nabla^2 v_1\|_{L^2}\|\nabla v_2\|_{L^\infty}+C\|\nabla v_1\|_{L^\infty}\|\nabla^2 v_2\|_{L^2}\\
%%\leq &\;C\|v_1\|_{D(A)}\|v_2\|_{D(A^{1+s/2})}+C\| v_1\|_{D(A^{1+s/2})}\|v_2\|_{D(A)}.
%\leq &\;C\|v_1\|_{D(A^{9/8})}\|v_2\|_{D(A^{9/8})}.
%\end{split}
%\label{eqn: D(A) estimate of source term}
%\end{equation}
%Here in the last inequality, although it is unnecessary, we put the bound to be the same as in the 3-D case with the hope that future derivations can be more concise.
%By interpolation, this
which implies that
\begin{equation}
\|\mathcal{P}^\alpha\mathcal{U}^\alpha(u^N,u^N)\|_{D(A)}
\leq C\|u^N\|_{D(A)}\|u^N\|_{D(A^{1+s/2})}.
\label{eqn: D(A) estimate of source term}
\end{equation}
Hence,
\begin{equation}
%\begin{split}
%&\;
|\langle A^2 u^N,\mathcal{U}^\alpha(u^N,u^N)\rangle|\\
%\leq C\|u^N\|_{D(A)}\|\mathcal{U}^\alpha(u^N,u^N)\|_{D(A)}
%\leq &\;C\|u^N\|_{D(A)}\|\nabla u^N\cdot (\nabla u^N)^T+\nabla u^N\cdot \nabla u^N-(\nabla u^N)^T\cdot \nabla u^N\|_{H^1}\\
%\leq &\;C\|u^N\|_{D(A)}\| \nabla u^N\|_{H^1}\|\nabla u^N\|_{H^{\frac{n}{2}+\epsilon}}\\
\leq C\|u^N\|_{D(A)}^2\|u^N\|_{D(A^{1+s/2})}.
%\end{split}
\label{eqn: higher order estimate for the source term}
\end{equation}
Note that when $\Omega$ is a smooth bounded domain, $\|v\|_{D(A^r)}\leq C\|A^r v\|_{L^2}$ for some constant $C$ depending on $r>0$ and $\Omega$.
Combining \eqref{eqn: low order energy estimate}-%\eqref{eqn: higher order estimate for inner product of the transport term}, % and
\eqref{eqn: higher order estimate for the source term}, % with \eqref{eqn: low order energy estimate} and \eqref{eqn: higher energy estimate crude form},
\begin{equation}
\frac{1}{2}\frac{d}{dt}\|u^N\|_{D(A)}^2 + \nu \|A^{1+s/2}u^N\|_{L^2}^2
\leq C\|u^N\|_{D(A)}^2\|A^{1+s/2}u^N\|_{L^2}.
\label{eqn: higher order energy estimate for u^N}
\end{equation}
%By the definition of $B$, %\eqref{eqn: estimate of the local solution},
%\begin{equation*}
%\begin{split}
%&\;\frac{d}{dt}\|w^N\|_{D(A)}^2 + \nu \|A^{1+s/2}w^N\|_{L^2}^2\\
%\leq &\;CM\|w^N\|_{D(A)}^2 + \frac{CM^2}{\nu}\|w^N\|_{D(A)}^2+CM\| w^N\|_{D(A)}\| u\|_{D(A^{1+s/2})}.
%\end{split}
%\end{equation*}
By Young's inequality,
\begin{equation*}
\frac{d}{dt}\|u^N\|_{D(A)}^2 + \nu \|A^{1+s/2}u^N\|_{L^2}^2
\leq C\nu^{-1}\|u^N\|_{D(A)}^4.
\end{equation*}
Taking time integral yields that
\begin{equation*}
(\|u^N\|_{L^\infty_T( D(A))}^2-\|\mathcal{P}_N u_0\|_{D(A)}^2 )+ \nu\int_0^T \|A^{1+s/2}u^N\|_{L^2}^2\,dt\leq C_*T\|u^N\|_{L^\infty_T(D(A))}^4.
\end{equation*}
Here $C_* = C_*(\alpha,s, n,\nu,\Omega)$.
By the continuity of $u^N$ in time, if $T$ is taken to be sufficiently small, which only depends on $\|u_0\|_{D(A)}$ and the constant $C_*$ but not on $N$, we deduce that $u^N$ exists on $[0,T]$ by a continuation argument if needed, and
\begin{equation}
\|u^N\|_{L^\infty_T (D(A))}^2+ \nu\int_0^T \|A^{1+s/2}u^N\|_{L^2}^2\,dt\\
\leq C\|u_0\|_{D(A)}^2.
\label{eqn: higher order energy estimate for u^N without T}
\end{equation}
Here $C$ is a universal constant.
%Applying this bound on the right hand side of \eqref{eqn: energy estimate for w^N with T}, and taking $T$ even smaller if need, we may end up with
%\begin{equation*}
%\|w^N\|_{L^\infty_T( D(A))}^2+ \nu\int_0^T \|A^{1+s/2}w^N\|_{L^2}^2\,dt\\
%\leq 2M^2
%\end{equation*}

\end{step}

\begin{step}
Since the bound in \eqref{eqn: higher order energy estimate for u^N without T} is uniform in $N$, there exists $u\in L^\infty_T (D(A))\cap L^2_T (D(A^{1+s/2}))$ satisfying \eqref{eqn: estimate of the local solution},
%\begin{equation}
%\|w\|_{L^\infty_T( D(A))}^2+ \nu\int_0^T \|A^{1+s/2}w\|_{L^2}^2\,dt\\
%\leq 2M^2,
%\end{equation}
such that up to a subsequence, $u^N$ weak-$*$ converges to $u$ in $L^\infty_T( D(A))$ and weakly in $L^2_T (D(A^{1+s/2}))$.

Next we derive an estimate for $\partial_t u^N$.
For arbitrary $v_1,v_2 \in D(A^{1+s/2})$,
\begin{equation}
\|\mathcal{P}^\alpha(v_1\cdot \nabla v_2)\|_{D(A^{1-s/2})}\leq C\|v_1\|_{D(A)}\|\nabla v_2\|_{H^{2-s}}\leq C\|v_1\|_{D(A)}\|v_2\|_{D(A^{1+s/2})}.
\label{eqn: higher order estimate for the transport term}
\end{equation}
Note that in the last inequality, we needed $s\geq 1/2$.
Combining this with \eqref{eqn: fixed point problem in finite dim space} and \eqref{eqn: D(A) estimate of source term}, we use boundedness of $\mathcal{P}_N$ and $\mathcal{P}^\alpha$ to derive that
\begin{equation*}
\begin{split}
\|\partial_t u^N\|_{D(A^{1-s/2})}\leq &\;C\|A^s u^N\|_{D(A^{1-s/2})}+ C\|u^N\|_{D(A)}\|u^N\|_{D(A^{1+s/2})}\\
\leq &\;C\|A^{1+s/2} u^N\|_{L^2}(1+ \|u^N\|_{D(A)}).
\end{split}
\end{equation*}
Thanks to \eqref{eqn: higher order energy estimate for u^N without T}, this implies that $\partial_t u^N$ has a uniform-in-$N$ bound in $L^2_T (D(A^{1-s/2}))$.
By interpolation and the Aubin-Lions Lemma \cite{temam1984navier}, $u^N\to u$ strongly in $L^p_T(D(A))$ for all $p\in [1,\infty)$.
This together with the weak convergence $u^N\rightharpoonup u$ in $L^2_T(D(A^{1+s/2}))$ is then sufficient for passing to the limit $N\to \infty$ in \eqref{eqn: fixed point problem in finite dim space}, which implies that $u$ is a weak solution.
Arguing as above, $\partial_t u\in L^2_T D(A^{1-s/2})$.
By a classic argument \cite[Lemma 1.2 in Chapter III]{temam1984navier},
$u$ is almost everywhere equal to a continuous function valued in $D(A)$, i.e., $u\in C_{[0,T]}(D(A))\cap L^2_T (D(A^{1+s/2}))$.
Moreover, $u$ satisfies the initial condition in \eqref{eqn: fixed point problem in finite dim space} since $\mathcal{P}_N u_0\to u_0$ strongly in $D(A)$.

\end{step}

\begin{step}
It remains to show the uniqueness.
Suppose there are two solutions $u_1$ and $u_2$ for \eqref{eqn: equivalent formulation of LANS equation} satisfying \eqref{eqn: estimate of the local solution}.
Define $w = u_1-u_2$.
Then $w\in C_{[0,T]}(D(A))\cap L^2_T (D(A^{1+s/2}))$ solves
\begin{equation*}
\partial_t w +\nu A^s w +\mathcal{P}^\alpha [w\cdot\nabla u_2 +\mathcal{U}^\alpha(w,u_2)]+\mathcal{P}^\alpha [u_1\cdot\nabla w +\mathcal{U}^\alpha(u_1,w)] = 0
\end{equation*}
with zero initial condition.
Similar to \eqref{eqn: low order energy estimate crude form}-\eqref{eqn: higher order energy estimate for u^N}, we derive energy estimates for $w$,
\begin{equation*}
\begin{split}
&\;\frac{1}{2}\frac{d}{dt}\|w\|_{L^2}^2 +\nu \|A^{s/2}w\|_{L^2}^2 \\
\leq &\;|\langle w,w\cdot \nabla u_2+\mathcal{U}^\alpha(w,u_2)\rangle|+|\langle w,\mathcal{U}^\alpha(u_1,w)\rangle|\\
\leq &\;C(\|u_1\|_{D(A)}+\|u_2\|_{D(A)})\|w\|_{D(A)}^2,
\end{split}
\end{equation*}
and
\begin{equation*}
\begin{split}
&\;\frac{1}{2}\frac{d}{dt}\|Aw\|_{L^2}^2 +\nu \|A^{1+s/2}w\|_{L^2}^2 \\
\leq &\;|\langle A^{1+s/2}w,A^{1-s/2}\mathcal{P}^\alpha(w\cdot \nabla u_2)\rangle|+|\langle Aw,(-\Delta)(u_1\cdot \nabla w)\rangle|\\
&\;+|\langle A w,A\mathcal{P}^\alpha\mathcal{U}^\alpha(w,u_2)\rangle|+|\langle A w,A\mathcal{P}^\alpha\mathcal{U}^\alpha(u_1,w)\rangle|\\
\leq &\;C\|A^{1+s/2}w\|_{L^2}\|w\|_{D(A)}\| u_2\|_{D(A^{1+s/2})}\\
&\;+|\langle A w,(\nabla u_1)^T (Aw+\Delta w) -\Delta u_1\cdot \nabla w - 2\partial_k u_{1,j}\partial_{jk}w \rangle|\\
&\;+C\| w\|_{D(A)}\|w\|_{D(A^{1+s/2})}(\|u_1\|_{D(A^{1+s/2})}+\|u_2\|_{D(A^{1+s/2})})\\
%\leq &\;C\|A^{1+s/2}w\|_{L^2}\|w\|_{D(A)}\| u_2\|_{D(A^{1+s/2})}\\ %+C\| w\|_{D(A)}^2\|u_2\|_{D(A^{1+s/2})}\\
%&\;+\|\langle A w\|_{},(\nabla u_1)^T (Aw+\Delta w) -\Delta u_1\cdot \nabla w - 2\partial_k u_{1,j}\partial_{jk}w \rangle|\\
%&\;+C\| w\|_{D(A)}^2(\|u_1\|_{D(A^{1+s/2})}+\|u_2\|_{D(A^{1+s/2})}).%\\%
\leq &\;C\|A^{1+s/2} w\|_{L^2}\|w\|_{D(A)}(\|A^{1+s/2} u_1\|_{L^2}+\|A^{1+s/2} u_2\|_{L^2}).
\end{split}
\end{equation*}
Here we note that the derivation in \eqref{eqn: explore cancellation in the transport term} is originally applied to smooth functions, but it also works here for $w$ by an approximation argument.
To justify this, \eqref{eqn: higher order estimate for the transport term} will be needed.
We omit the details.

Combining these two estimates, by Young's inequality,
\begin{equation*}
\begin{split}
&\;\frac{d}{dt}\|w\|_{D(A)}^2 +\nu \|A^{1+s/2}w\|_{L^2}^2 \\
\leq &\;%C\|w\|_{D(A)}^2(\|u_1\|_{D(A^{1+s/2})}+\|u_2\|_{D(A^{1+s/2})})\\
%&\;+
C\nu^{-1}\|w\|_{D(A)}^2(\|A^{1+s/2} u_1\|_{L^2}^2+\|A^{1+s/2} u_2\|_{L^2}^2).
\end{split}
\end{equation*}
By Gronwall's inequality and \eqref{eqn: estimate of the local solution}, %for all $t\in [0,T]$,
\begin{equation*}
\|w\|_{L^\infty_T D(A)}^2+\nu\|A^{1+s/2}w\|_{L^2_T L^2}^2\leq C(\alpha, s, n, \Omega, \nu,T,\|u_0\|_{D(A)})\|w(0)\|_{D(A)}^2.
\end{equation*}
It follows that $w \equiv 0$ since $w(0) = 0$, and thus $u_1 \equiv u_2$.
This estimate also implies continuous dependence of the solution on the $D(A)$-initial data.

This completes the proof of the local well-posedness in the case of $\Omega$ being a smooth bounded domain.
\end{step}

\end{case}
\begin{case}
Now suppose $\Omega = \mathbb{R}^n$.
Let $\eta\in C_0^\infty(\mathbb{R}^n)$ be an even smooth mollifier, such that $\eta\geq 0$ is supported in the unit ball centered at $0$, and $\eta$ has integral $1$.
%For simplicity, we assume $\eta$ is even.
Define $\eta_\varepsilon(x) = \varepsilon^{-n} \eta(x/\varepsilon)$.
Let $u^\varepsilon$ solve
\begin{equation}
\partial_t u^\varepsilon + \nu \eta_\varepsilon*\eta_\varepsilon* A^s u^\varepsilon +\mathcal{P}^\alpha \eta_\varepsilon*[(\eta_\varepsilon*u^\varepsilon)\cdot \nabla (\eta_\varepsilon*u^\varepsilon)+\mathcal{U}^\alpha(\eta_\varepsilon*u^\varepsilon,\eta_\varepsilon*u^\varepsilon)]=0
\label{eqn: regularized problem in whole space}
\end{equation}
with initial data $u^\varepsilon|_{t = 0} = \eta_\varepsilon* u_0$.
We shall view \eqref{eqn: regularized problem in whole space} as an ODE $\partial_t u^\varepsilon = F_\varepsilon(u^\varepsilon)$ in $D(A)$, with
\begin{equation*}
F_\varepsilon(u^\varepsilon) = -\nu \eta_\varepsilon*\eta_\varepsilon* A^s u^\varepsilon -\mathcal{P}^\alpha \eta_\varepsilon*[(\eta_\varepsilon*u^\varepsilon)\cdot \nabla (\eta_\varepsilon*u^\varepsilon)+\mathcal{U}^\alpha(\eta_\varepsilon*u^\varepsilon,\eta_\varepsilon*u^\varepsilon)].
\end{equation*}
Thanks to \eqref{eqn: D(A) estimate of source term}, \eqref{eqn: higher order estimate for the transport term} and smoothness of $\eta_\varepsilon$, it is not hard to show that $F_\varepsilon(u^\varepsilon)$ is locally Lipschitz in $u^\varepsilon\in D(A)$.
Then local existence and uniqueness of $u^\varepsilon\in C^1_{T_\varepsilon}(D(A))$ follows from ODE theory on Banach spaces.

Then we derive energy estimates for $u^\varepsilon$.
Note that in the whole space case, $A^s$ applied to $u^\varepsilon\in D(A)$ is simply a Fourier multiplier, which thus commutes with the mollification by $\eta_\varepsilon$.
We proceed as in the bounded domain case to find that
\begin{equation*}
%\begin{split}
\frac{1}{2}\frac{d}{dt}\|u^\varepsilon\|_{L^2}^2+\nu \|\eta_\varepsilon* A^{s/2}u^\varepsilon\|_{L^2}
%\leq |\langle \eta_\varepsilon* u^\varepsilon, \mathcal{U}^\alpha(\eta_\varepsilon* u^\varepsilon,\eta_\varepsilon* u^\varepsilon)\rangle|
\leq C\|\eta_\varepsilon* u^\varepsilon\|_{L^2}\|\eta_\varepsilon* u^\varepsilon\|_{D(A)}^2,
%\end{split}
\end{equation*}
and
\begin{equation*}
%\begin{split}
%&\;
\frac{1}{2}\frac{d}{dt}\|Au^\varepsilon\|_{L^2}^2+\nu \|\eta_\varepsilon* A^{1+s/2}u^\varepsilon\|_{L^2}%\\
\leq C\|\eta_\varepsilon
*u^\varepsilon\|_{D(A)}^2\|\eta_\varepsilon
*u^\varepsilon\|_{D(A^{1+s/2})}.
\end{equation*}
We use $\|\eta_\varepsilon
*u^\varepsilon\|_{D(A^{1+s/2})}\leq \|\eta_\varepsilon
*u^\varepsilon\|_{D(A)}+\|\eta_\varepsilon
*A^{1+s/2}u^\varepsilon\|_{L^2}$ and Young's inequality to derive that
\begin{equation*}
%\begin{split}
%&\;
\frac{d}{dt}\|Au^\varepsilon\|_{L^2}^2+\nu \|\eta_\varepsilon* A^{1+s/2}u^\varepsilon\|_{L^2}%\\
\leq C\|\eta_\varepsilon
*u^\varepsilon\|_{D(A)}^3+C\nu^{-1}\|\eta_\varepsilon
*u^\varepsilon\|_{D(A)}^4.%2\|A^{1+s/2}\eta_\varepsilon*u^\varepsilon\|_{L^2}.
%\end{split}
\end{equation*}
Combining these estimates, by Young's inequality for convolutions,
\begin{equation*}
%\begin{split}
%&\;
\frac{d}{dt}\|u^\varepsilon\|_{D(A)}^2+\nu \|\eta_\varepsilon* A^{1+s/2}u^\varepsilon\|_{L^2}%\\
%\leq &\;C\|\eta_\varepsilon
%*u^\varepsilon\|_{D(A)}^3+C\nu^{-1}\|\eta_\varepsilon
%*u^\varepsilon\|_{D(A)}^4\\
\leq C\|u^\varepsilon\|_{D(A)}^3+C\nu^{-1}\|u^\varepsilon\|_{D(A)}^4.
%\end{split}
\end{equation*}
Taking time integral yields that
\begin{equation*}
\begin{split}
&\;(\|u^\varepsilon\|_{L^\infty_T( D(A))}^2-\|\eta_\varepsilon
*u_0\|_{D(A)}^2 )+ \nu\int_0^T \|\eta_\varepsilon* A^{1+s/2}
u^\varepsilon\|_{L^2}^2\,dt\\
\leq &\;C_*T(\|u^\varepsilon\|_{L^\infty_T(D(A))}^3+\nu^{-1}\|u^\varepsilon\|_{L^\infty_T(D(A))}^4).
\end{split}
\end{equation*}
Here $C_* = C_*(\alpha,s, n,\Omega)$.
By the continuity of $u^\varepsilon$ in time, if $T$ is taken to be sufficiently small, which only depends on $\|u_0\|_{D(A)}$, $\nu$ and the constant $C_*$ but not on $\varepsilon$, we deduce that $u^\varepsilon$ exists on $[0,T]$ by a continuation argument if needed, and
\begin{equation}
\|\eta_\varepsilon *u^\varepsilon\|_{L^\infty_T (D(A))}^2+ \nu\int_0^T \|\eta_\varepsilon* A^{1+s/2}u^\varepsilon\|_{L^2}^2\,dt\\
\leq C\|u_0\|_{D(A)}^2.
\label{eqn: higher order energy estimate for u^eps without T}
\end{equation}
Here $C$ is a universal constant.

To this end, since $\eta_\varepsilon*u^\varepsilon$ is uniformly bounded in $L^\infty_T(D(A))\cap L^2_T(D(A^{1+s/2}))$,
there exists $u\in L^\infty_T(D(A))\cap L^2_T(D(A^{1+s/2}))$ satisfying \eqref{eqn: estimate of the local solution}, such that up to a subsequence, $\eta_\varepsilon *u^\varepsilon$ weak-$*$ converges to $u$ in $L^\infty_T( D(A))$ and weakly in $L^2_T (D(A^{1+s/2}))$.
We argue as in the bounded domain case that %
%Next we derive an estimate of $\partial_t u^N$.
%For arbitrary $v_1,v_2 \in D(A^{1+s/2})$,
%\begin{equation}
%\|\mathcal{P}^\alpha(v_1\cdot \nabla v_2)\|_{D(A^{1-s/2})}\leq C\|v_1\|_{D(A)}\|v_2\|_{H^{3-s}}\leq C\|v_1\|_{D(A)}\|v_2\|_{D(A^{1+s/2})}.
%\label{eqn: higher order estimate for the transport term}
%\end{equation}
%Note that here we require $s\geq 1/2$.
%Combining this with \eqref{eqn: fixed point problem in finite dim space} and \eqref{eqn: D(A) estimate of source term}, we use boundedness of $\mathcal{P}_N$ and $\mathcal{P}^\alpha$ to derive that
\begin{equation*}
%\begin{split}
\|\partial_t u^\varepsilon\|_{D(A^{1-s/2})}%\leq &\;C\|A^s u^\varepsilon\|_{D(A^{1-s/2})}+ C\|u^\varepsilon\|_{D(A)}\|u^\varepsilon\|_{D(A^{1+s/2})}\\
\leq C\|\eta_\varepsilon* u^\varepsilon\|_{D(A^{1+s/2})}(1+ \|\eta_\varepsilon*u^\varepsilon\|_{D(A)}).
%\end{split}
\end{equation*}
By \eqref{eqn: higher order energy estimate for u^eps without T}, this implies that $\partial_t (\eta_\varepsilon* u^\varepsilon)$ has a uniform-in-$\varepsilon$ bound in $L^2_T (D(A^{1-s/2}))$.
By interpolation and the Aubin-Lions Lemma \cite{temam1984navier}, $\eta_\varepsilon* u^\varepsilon\to u$ strongly in $L^p_T(D(A))$ for all $p\in [1,\infty)$.
This together with the weak convergence $u^\varepsilon\rightharpoonup u$ in $L^2_T(D(A^{1+s/2}))$ is then sufficient for passing to the limit $\varepsilon\to 0$ in \eqref{eqn: regularized problem in whole space} mollified by $\eta_\varepsilon$, which implies that $u$ is a weak solution.
Time continuity of $u$ in $D(A)$ can be justified as before.
So are the initial condition and the uniqueness of $u$.
\end{case}

This completes the proof.
\end{proof}
\end{proposition}

Now we can prove global well-posedness by combining Proposition \ref{prop: local existence and uniqueness} with a global $H^1$-energy estimate.
%A special consideration is given to the whole space case, where $\|\cdot\|_{D(A^r)}$ and $\|A^r\cdot\|_{L^2}$ are not equivalent.

\begin{proof}[Proof of Theorem \ref{thm: global existence and uniqueness}]
Take the local solution $u_*\in C_{[0,T]}(D(A))\cap L^2_T(D(A^{1+s/2}))$ that solves
\begin{equation}
\partial_t u_* + \nu A^s u_* = -\mathcal{P}^\alpha[u_*\cdot\nabla u_*+\mathcal{U}^\alpha(u_*,u_*)]\mbox{ in }\Omega\times[0,T],\quad u_*|_{t = 0} = u_0.
\label{eqn: equation of the local solution}
\end{equation}
%Given the local solution, the derivation in \eqref{eqn: projected equation for u}-\eqref{eqn: v equation} can be made rigorous since the Leray projection is well-defined for $L^2(\Omega)$-functions \cite{temam1984navier}.
%
Take inner product of $(1-\alpha^2 \Delta)u_*$ and \eqref{eqn: equation of the local solution}.
It is valid to do so since $u_*\in C_{[0,T]}(D(A))\cap L^2_T(D(A^{1+s/2}))$ while the right hand side is in $ L^2_T D(A^{1-s/2})$, which has been shown in the proof of Proposition \ref{prop: local existence and uniqueness}.
Taking integration by parts,
\begin{equation}
\begin{split}
&\;\frac{1}{2}\frac{d}{dt}(\|u_*\|^2_{L^2(\Omega)} + \alpha^2\|A^{1/2}u_*\|^2_{L^2(\Omega)}) + \nu(\|A^{s/2}u_*\|^2_{L^2(\Omega)} + \alpha^2\|A^{(1+s)/2}u_*\|^2_{L^2(\Omega)})\\
%=&\;\langle (1 - \alpha^2 \Delta)u_*, f(u_*,u_*)\rangle\\
=&\;-\langle (1 - \alpha^2 \Delta)u_*, [u_*\cdot \nabla u_* + \mathcal{U}^\alpha(u_*,u_*)+(1-\alpha^2 \Delta)^{-1}\nabla q]\rangle\\
=&\;-\langle (1 - \alpha^2 \Delta)u_*, u_*\cdot \nabla u_*\rangle - \langle u_*, \alpha^2 \mathrm{div}\,[\nabla u_*\cdot \nabla u_*^T+\nabla u_*\cdot \nabla u_* - \nabla u_*^T\cdot \nabla u_*]\rangle\\
=&\;\alpha^2 \langle \Delta u_*, u_*\cdot \nabla u_*\rangle + \alpha^2\langle \partial_j u_*^i,  \partial_k u_*^i\partial_k u_*^j+\partial_k u_*^i\partial_j u_*^k - \partial_i u_*^k\partial_j u_*^k\rangle\\
%=&\;\alpha^2 \langle \partial_{jj} u_*^i, u_*^k\partial_k u_*^i\rangle + \alpha^2\langle \partial_j u_*^i, \partial_k u_*^i\partial_j u_*^k\rangle\\
=&\;-\alpha^2 \langle \partial_{j} u_*^i, \partial_j u_*^k\partial_k u_*^i\rangle -\alpha^2 \langle \partial_{j} u_*^i,  u_*^k\partial_{kj} u_*^i\rangle + \alpha^2\langle \partial_j u_*^i, \partial_k u_*^i\partial_j u_*^k\rangle = 0.
\end{split}
\label{eqn: global H^1 energy estimate}
\end{equation}
%Here $\langle \cdot,\cdot\rangle$ denotes the $L^2$-inner product on $\Omega$.
%Hence, %for all $t\in [0,T]$,
By a limiting argument, this implies that for all $t \in [0,T]$,
\begin{equation}
\begin{split}
&\;(\|u_*\|_{L^2}^2+\alpha^2 \|A^{1/2}u_*\|_{L^2}^2)(t)\\
&\;\quad +2\nu\int_0^t(\|A^{s/2}u_*\|^2_{L^2} +\alpha^2\|A^{(1+s)/2}u_*\|^2_{L^2})(\tau)\,d\tau\\
\leq &\;\|u_0\|_{L^2}^2+\alpha^2 \|A^{1/2}u_0\|_{L^2}^2.
\end{split}
\label{eqn: L^2 energy estimate}
\end{equation}

On the other hand, it holds in the scalar distribution sense on $(0,T)$ that \cite[Lemma 1.2 in Chapter III]{temam1984navier}
\begin{equation*}
\frac{1}{2}\frac{d}{dt}\|Au_*\|_{L^2}^2 + \nu \|A^{1+s/2} u_*\|_{L^2}^2 = -\langle A^2 u_*, u_*\cdot \nabla u_*+\mathcal{U}^\alpha(u_*,u_*)\rangle.
\end{equation*}
%Here $\langle\cdot,\cdot\rangle$ denotes $L^2$-inner product on $\Omega$.
By a limiting argument and the continuity of $u_*$ in $D(A)$, for all $t\in[0,T]$,
\begin{equation}
\begin{split}
 &\;\|Au_*\|_{L^2}^2(t) + 2\nu\int_0^t \|A^{1+s/2}u_*\|_{L^2}^2\,d\tau\\
= &\;\|Au_0\|_{L^2}^2 - 2\int_0^t \langle A^2 u_*, u_*\cdot \nabla u_*+\mathcal{U}^\alpha(u_*,u_*)\rangle(\tau)\,d\tau.
\end{split}
\label{eqn: energy estimate of the D(A) semi-norm}
\end{equation}
Once again, the derivation in \eqref{eqn: explore cancellation in the transport term} also work for $u_*$ here by an approximation argument.
It will be used below to bound the integrand.

To this end, we proceed in two cases.
\setcounter{case}{0}
\begin{case}
Suppose $\Omega$ is a bounded smooth domain.
In this case, the norm $\|u_*\|_{D(A^r)}$ is equivalent to the seminorm $\|A^ru_* \|_{L^2}$ for all $r>0$, as all the $\mu_j$'s are positive.
See Section \ref{section: preliminaries}.

Since $s\geq n/4$, by \eqref{eqn: higher order estimate for inner product of the transport term},
\begin{equation}
|\langle A^2 u_*, u_*\cdot \nabla u_*\rangle|\leq C\| u_*\|_{D(A^{1+s/2})}\|u_*\|_{D(A^{(1+s)/2})}\|u_*\|_{D(A)}.
\label{eqn: improved higher order estimate for transport term}
\end{equation}
%In order to bound the second term in the integral of \eqref{eqn: energy estimate of the D(A) semi-norm}, we need an improved version of \eqref{eqn: D(A) estimate of source term} and \eqref{eqn: D(A) estimate of source term}.
Likewise, by \eqref{eqn: D(A) estimate of source term v1 v2},
\begin{equation*}
\|\nabla u_*\cdot (\nabla u_*)^T+\nabla u_*\cdot \nabla u_*-(\nabla u_*)^T\cdot \nabla u_*\|_{H^1}\leq C\|u_*\|_{D(A^{1+s/2})}\|u_*\|_{D(A^{(1+s)/2})},
\end{equation*}
and thus
\begin{equation*}
|\langle A^2 u_*, \mathcal{U}^\alpha(u_*,u_*)\rangle|\leq C\|u_*\|_{D(A)}\| u_*\|_{D(A^{1+s/2})}\|u_*\|_{D(A^{(1+s)/2})}.
\end{equation*}
Combining this with \eqref{eqn: L^2 energy estimate}, \eqref{eqn: energy estimate of the D(A) semi-norm} and \eqref{eqn: improved higher order estimate for transport term}, we obtain that
\begin{equation*}
\begin{split}
&\;(\|u_*\|_{D(A)}^2+\alpha^2 \|A^{1/2}u_*\|_{L^2}^2)(t)\\
&\;\quad +2\nu\int_0^t(\|A^{s/2}u_*\|^2_{D(A)} +\alpha^2\|A^{(1+s)/2}u_*\|^2_{L^2})(\tau)\,d\tau\\
\leq  &\;\|u_0\|_{D(A)}^2+\alpha^2 \|A^{1/2}u_0\|_{L^2}^2 \\
&\;+C\int_0^t (\|u_*\|_{D(A)}\| u_*\|_{D(A^{1+s/2})}\|u_*\|_{D(A^{(1+s)/2})})(\tau)\,d\tau.
\end{split}
%\label{eqn: energy estimate of the D(A) semi-norm}
\end{equation*}
Recall that $\|u_*\|_{D(A^{r})}\leq C\|A^{r}u_*\|_{L^2}$.
By Young's inequality and \eqref{eqn: L^2 energy estimate},
\begin{equation*}
\begin{split}
&\;\int_0^t (\|u_*\|_{D(A)}\| u_*\|_{D(A^{1+s/2})}\|u_*\|_{D(A^{(1+s)/2})})(\tau)\,d\tau\\
%&\;(\|u_*\|_{D(A)}^2+\alpha^2 \|A^{1/2}u_*\|_{L^2}^2)(t)\\
%&\;\quad +2\nu\int_0^t(\|A^{s/2}u_*\|^2_{D(A)} +\alpha^2\|A^{(1+s)/2}u_*\|^2_{L^2})(\tau)\,d\tau\\
\leq % &\;\|u_0\|_{D(A)}^2+\alpha^2 \|A^{1/2}u_0\|_{L^2}^2 \\
%&\;C\int_0^t (\|u_*\|_{D(A)}\|u_*\|_{L^2}\|u_*\|_{D(A^{(1+s)/2})})(\tau)\,d\tau\\
&\;C\int_0^t (\|u_*\|_{D(A)}\|A^{1+s/2}u_*\|_{L^2}\|A^{(1+s)/2}u_*\|_{L^2})(\tau)\,d\tau\\
%&\;+\nu\alpha^2\int_0^t \|A^{1+s/2}u_*\|_{L^2}^2(\tau)\,d\tau\\
%&\;+C\nu^{-1}\alpha^{-2}\int_0^t \|u_*\|_{D(A)}^2(\tau)(\|u_*\|_{L^2}^2(\tau)+\|A^{(1+s)/2}u_*\|_{L^2}^2(\tau))\,d\tau\\
\leq  %&\;\|u_0\|_{D(A)}^2+\alpha^2 \|A^{1/2}u_0\|_{L^2}^2 \\
%&\;\nu\alpha^2\int_0^t \|A^{(1+s)/2}u_*\|_{L^2}^2(\tau)\,d\tau
%+C(1+\nu^{-1}\alpha^{-2})\int_0^t \|u_*\|_{D(A)}^2(\tau)\,d\tau\\
&\;\nu\int_0^t \|A^{1+s/2}u_*\|_{L^2}^2(\tau)\,d\tau
+C\nu^{-1}\int_0^t \|u_*\|_{D(A)}^2(\tau)\|A^{(1+s)/2}u_*\|_{L^2}^2(\tau)\,d\tau.
\end{split}
\end{equation*}
Hence,
\begin{equation*}
\begin{split}
&\;(\|u_*\|_{D(A)}^2+\alpha^2 \|A^{1/2}u_*\|_{L^2}^2)(t)\\
&\;\quad +\nu\int_0^t(\|A^{s/2}u_*\|^2_{D(A)} +\alpha^2\|A^{(1+s)/2}u_*\|^2_{L^2})(\tau)\,d\tau\\
%\leq  &\;\|u_0\|_{D(A)}^2+\alpha^2 \|A^{1/2}u_0\|_{L^2}^2 \\
%&\;+\nu\alpha^2\int_0^t \|A^{(1+s)/2}u_*\|_{L^2}^2(\tau)\,d\tau
%+C(1+\nu^{-1}\alpha^{-2})\int_0^t \|u_*\|_{D(A)}^2(\tau)\,d\tau\\
%&\;+\nu\int_0^t \|A^{1+s/2}u_*\|_{L^2}^2(\tau)\,d\tau\\
%&\;+C\nu^{-1}\int_0^t \|u_*\|_{D(A)}^2(\tau)(\|u_*\|_{L^2}^2(\tau)+\|A^{(1+s)/2}u_*\|_{L^2}^2(\tau))\,d\tau\\
\leq  &\;\|u_0\|_{D(A)}^2+\alpha^2 \|A^{1/2}u_0\|_{L^2}^2 +C\int_0^t \|u_*\|_{D(A)}^2(\tau)\|A^{(1+s)/2}u_*\|_{L^2}^2(\tau)\,d\tau.
\end{split}
%\label{eqn: energy estimate of the D(A) semi-norm}
\end{equation*}
where $C= C(\alpha,s,n,\Omega,\nu)$.
Since  $\|A^{(1+s)/2}u_*\|_{L^2}\in L^2([0,T])$ for all $T\in [0,+\infty)$ by \eqref{eqn: L^2 energy estimate} with uniform-in-$T$ bound, then a uniform-in-$T$ global bound for $\|u_*\|_{D(A)}$ follows from the Gronwall's inequality.
Global well-posedness can be proved by the local well-posedness and a continuation argument, and \eqref{eqn: estimate of the global solution} follows also from the last inequality.
\end{case}

\begin{case}
Now suppose $\Omega = \mathbb{R}^n$.
In this case, we shall slightly change the argument so that the final estimate will be uniform in $T>0$.
We shall take advantage of $A^r u_* = (-\Delta)^r u_*$ in this case.

Again by \eqref{eqn: higher order estimate for inner product of the transport term},
\begin{equation*}
|\langle A^2 u_*, u_*\cdot \nabla u_*\rangle|
\leq C\| Au_*\|_{H^s}\|\nabla u_*\|_{H^{\frac{n}{2}-s}}\|u_*\|_{\dot{H}^2},
%\label{eqn: improved higher order estimate for transport term whole space}
\end{equation*}
%In order to bound the second term in the integral of \eqref{eqn: energy estimate of the D(A) semi-norm}, we need an improved version of \eqref{eqn: D(A) estimate of source term} and \eqref{eqn: D(A) estimate of source term}.
and by \eqref{eqn: D(A) estimate of source term v1 v2},
%\begin{equation*}
%\|\nabla u_*\cdot (\nabla u_*)^T+\nabla u_*\cdot \nabla u_*-(\nabla u_*)^T\cdot \nabla u_*\|_{H^1}\leq C\|u_*\|_{D(A^{1+s/2})}\|u_*\|_{D(A^{(1+s)/2})},
%\end{equation*}
%and thus
\begin{equation*}
|\langle A^2 u_*, \mathcal{U}^\alpha(u_*,u_*)\rangle|
\leq C\|Au_*\|_{L^2}\| \nabla^2 u_*\|_{H^s}\|\nabla u_*\|_{H^{\frac{n}{2}-s}}.
\end{equation*}
Combining them with \eqref{eqn: L^2 energy estimate} and \eqref{eqn: energy estimate of the D(A) semi-norm}, by Young's inequality, we obtain that
\begin{equation}
\begin{split}
&\;(\|u_*\|_{D(A)}^2+\alpha^2 \|A^{1/2}u_*\|_{L^2}^2)(t)\\
&\;\quad +2\nu\int_0^t(\|A^{s/2}u_*\|^2_{D(A)} +\alpha^2\|A^{(1+s)/2}u_*\|^2_{L^2})(\tau)\,d\tau\\
\leq  &\;\|u_0\|_{D(A)}^2+\alpha^2 \|A^{1/2}u_0\|_{L^2}^2 \\
&\;+C\int_0^t (\|Au_*\|_{L^2}\| Au_*\|_{D(A^{s/2})}\|A^{1/2} u_*\|_{D(A^{s/2})})(\tau)\,d\tau.
%\leq  &\;\|u_0\|_{D(A)}^2+\alpha^2 \|A^{1/2}u_0\|_{L^2}^2 \\
%&\;+C\int_0^t (\|Au_*\|_{L^2}\| Au_*\|_{D(A^{s/2})}(\|A^{s/2} u_*\|_{L^2}+\|A^{(1+s)/2} u_*\|_{L^2}))(\tau)\,d\tau.
\end{split}
\label{eqn: higher energy estimate whole space}
\end{equation}
By interpolation,
\begin{equation*}
\| Au_*\|_{D(A^{s/2})} \leq C\| A^{s/2}u_*\|_{D(A)},\quad \|A^{1/2} u_*\|_{D(A^{s/2})}\leq C\|A^{s/2} u_*\|_{D(A^{1/2})}.
\end{equation*}
By Young's inequality, \eqref{eqn: higher energy estimate whole space} then becomes
\begin{equation}
\begin{split}
&\;(\|u_*\|_{D(A)}^2+\alpha^2 \|A^{1/2}u_*\|_{L^2}^2)(t)\\
&\;+2\nu\int_0^t(\|A^{s/2}u_*\|^2_{D(A)} +\alpha^2\|A^{(1+s)/2}u_*\|^2_{L^2})(\tau)\,d\tau\\
\leq  &\;\|u_0\|_{D(A)}^2+\alpha^2 \|A^{1/2}u_0\|_{L^2}^2 \\
&\;+C\int_0^t (\|Au_*\|_{L^2}\| A^{s/2}u_*\|_{D(A)}\|A^{s/2} u_*\|_{D(A^{1/2})})(\tau)\,d\tau\\
\leq  &\;\|u_0\|_{D(A)}^2+\alpha^2 \|A^{1/2}u_0\|_{L^2}^2 +\nu\int_0^t \| A^{s/2}u_*\|_{D(A)}^2(\tau)\,d\tau\\
&\;+C\nu^{-1}\int_0^t \|u_*\|_{D(A)}^2(\tau)\|A^{s/2} u_*\|_{D(A^{1/2})}^2(\tau)\,d\tau.
\end{split}
\label{eqn: higher order energy estimate whole space simplified}
\end{equation}
By \eqref{eqn: L^2 energy estimate}, $\|A^{s/2}u_*\|_{L^2_TD(A^{1/2})}\leq C(\alpha, \nu, \|u_0\|_{D(A^{1/2})})$.
Hence, by the Gronwall's inequality, we obtain a uniform-in-time global bound for $\|u_*\|_{D(A)}$, i.e., for all $t\in [0,T]$,
\begin{equation*}
\|u_*\|_{D(A)}(t)\leq C(\alpha,s,n,\nu,\|u_0\|_{D(A^{1/2})}).
\end{equation*}
In particular, this bound does not rely on $T$.
Then the global well-posedness and \eqref{eqn: estimate of the global solution} follow as before.
\end{case}
\end{proof}

\section{Improved Regularity of $u_*$}\label{section: improved regularity}
In this section, we shall show that the global solution $u_*$ gains regularity instantaneously when $t>0$.
We proceed in two different cases.

\subsection{Non-critical case: $s>1/2$}
For simplicity, denote $f(u_1,u_2)= -\mathcal{P}^\alpha[u_1\cdot \nabla u_2 + \mathcal{U}^\alpha(u_1,u_2)]$.
In the proof of Proposition \ref{prop: local existence and uniqueness}, we have derived an $D(A^{1-s/2})$-estimate for $f$ (see \eqref{eqn: D(A) estimate of source term v1 v2}, \eqref{eqn: D(A) estimate of source term} and \eqref{eqn: higher order estimate for the transport term}).
%By the proof of Proposition \ref{prop: local existence and uniqueness}, we have already established the following estimate.
Yet, the following lemma is still useful.
\begin{lemma}\label{lemma: unified spatial estimate of f}
For all $r\in(\frac{n}{2},2]$, %and $r'>n/2$ and $r'\geq r-1$,
\begin{equation*}
\|f(u_1,u_2)\|_{D(A^{r/2})}\leq C\|u_1\|_{D(A)}\|A^{1/2} u_2\|_{D(A^{r/2})},%+\|A^{(r+1)/2} u_2\|_{L^2}+\|A^{(r'+1)/2} u_2\|_{L^2}),
\end{equation*}
where $C = C(\alpha, r, n,\Omega)$.
%If $\Omega = \mathbb{R}^n$, it holds for all $r\in [1,+\infty)$.
\begin{proof}
%Since $H^{s_1}(\Omega)\cdot H^{s_2}(\Omega) \hookrightarrow H^{\min\{s_1,s_2\}}(\Omega)$ as long as $\max\{s_1,s_2\}>n/2$, using  we have that
The proof is straightforward.
By the boundedness of $\mathcal{P}^\alpha$ and Remark \ref{rmk: relation between D(A^r) and V^2r},
\begin{equation}
\begin{split}
&\;\|f(u_1,u_2)\|_{D(A^{r/2})} \\
\leq &\;C(\|u_1\cdot\nabla u_2\|_{H^{r}}+\|\nabla u_1\cdot \nabla u_2^T+\nabla u_1\cdot \nabla u_2- \nabla u_1 ^T\cdot \nabla u_2\|_{H^{r-1}})\\
\leq &\;C(\|u_1\|_{H^2}\|\nabla u_2\|_{H^{r}}+\|\nabla u_1\|_{H^{r-1}}\|\nabla u_2\|_{H^{r}})\\
\leq &\;C\|u_1\|_{D(A)}(\|A^{1/2} u_2\|_{L^2}+\|A^{(r+1)/2} u_2\|_{L^2}).
\end{split}
\label{eqn: derivation of the unified spatial estimate of f}
\end{equation}
\end{proof}
\end{lemma}

%\begin{lemma}\label{lemma: estimte of f_u}
%Assume $T>0$ and the assumptions \eqref{eqn: assumption on the regularity of the domain} and \eqref{eqn: assumption on s}.
%For all $u_1,u_2\in L_T^\infty(D(A))\cap L_T^2(D(A^{1+s/2}))$,
%\begin{equation}
%\begin{split}
%\|f(u_1,u_2)\|_{L_T^2(D(A^{1-s/2}))}\leq &\;CT^{1- 1/(2s)}\|u_1\|_{L_T^\infty(D(A))}\|u_2\|_{L_T^\infty(D(A))}^{(2s-1)/s}\|u_2\|_{L_T^2(D(A^{1+s/2}))}^{(1-s)/s}\\
%&\;+CT^{1/8}\|u_1\|_{L_T^\infty(D(A))}\|u_2\|_{L_T^\infty(D(A))}^{1/4}\|u_2\|_{L_T^2(D(A^{1+s/2}))}^{3/4},
%\end{split}
%\label{eqn: estimate of f_u}
%\end{equation}
%where $C = C(\alpha,s,n,\Omega)$ is a universal constant.
%\begin{proof}
%We first apply Lemma \ref{lemma: unified spatial estimate of f} with $r = 2-s$ and $r' = 1+3s/4>n/2$ to find that
%\begin{equation*}
%\begin{split}
%\|f(u_1,u_2)\|_{D(A^{1-s/2})}\leq &\;C\|u_1\|_{D(A)}(\|u_2\|_{D(A^{(3-s)/2})}+\|u_2\|_{D(A^{(1+3s/8)})})\\
%\leq &\;C\|u_1\|_{D(A)}\left[\|u_2\|_{D(A)}^{(2s-1)/s}\|u_2\|_{D(A^{1+s/2})}^{(1-s)/s}+C\|u_2\|_{D(A)}^{1/4}\|u_2\|_{D(A^{1+s/2})}^{3/4}\right].
%\end{split}
%%\label{eqn: spatial estimate of f in the space-time estimate of f}
%\end{equation*}
%Here we used interpolation in the last inequality.
%Taking integral in time, we obtain \eqref{eqn: estimate of f_u} by H\"{o}lder's inequality.
%\end{proof}
%\end{lemma}

Lemma \ref{lemma: unified spatial estimate of f} roughly shows that the regularity of $f(u_*,u_*)$ is one order lower than that of $u_*$. % if we can take $r= r' \in(n/2,2]$ although the estimate is nonlinear; o
Since the backbone equation $\partial_t u_* +\nu A^s u_* = f$ (i.e., \eqref{eqn: equivalent formulation of LANS equation}) implies that $u_*$ admits regularity $2s$-order higher than $f$, we immediately obtain improved regularity of $u_*$ when $s>1/2$ by bootstrapping, which is why we call the case $s>1/2$ non-critical, and which leads to the following theorem.

\begin{theorem}[Improved regularity of $u_*$ in the non-critical case]\label{thm: improved regularity in non-critical case}
If $s>1/2$, the global solution $u_*$ obtained in Theorem \ref{thm: global existence and uniqueness} satisfies that for all $r\in[0,s/2]$, and all $t>0$,
\begin{equation}
\|u_*(t)\|_{D(A^{1+r})}\leq C(t^{-\frac{r}{s}}+1)\|u_0\|_{D(A)},
\label{eqn: estimate of the improved regularity of u_*}
\end{equation}
where $C = C(\alpha,  s, n, \nu,\Omega, \|u_0\|_{D(A)})$.
In particular, when $\|u_0\|_{D(A)}\to 0$, $C$ converges to a universal constant depending on $\alpha$, $s$, $n$, $\nu$ and $\Omega$.

\begin{proof}
Define $r_k = 2+s+(2s-1)k$ for $k = 1,\cdots , K$, where $K = \left[\frac{1-s}{2s-1}\right]$ is taken in the way that $r_K\leq 3$ while $r_{K+1}>3$.
%
%
%For given $k\in \mathbb{N}$ sufficiently small such that $r_k-1 \in [1,2]$,
We apply Lemma \ref{lemma: unified spatial estimate of f} with $r = r_k-1$ and use interpolation to find that
\begin{equation}
\|f(u_*,u_*)\|_{D(A^{(r_k-1)/2})}\leq C(\alpha, s, n, \Omega)\|u_*\|_{D(A)}\|A^{s/2}u_*\|_{D(A^{(r_k-s)/2})}.
\label{eqn: estimate of f_u in the noncritical case}
\end{equation}
%where  is the largest $k$-value that can be achieved.
%The reason why we need $r_k-1 \in [1,2]$ is that we wish to keep the first factor on the right hand side of \eqref{eqn: estimate of f_u in the noncritical case} to be $\|u_*\|_{D(A)}$ instead of any other higher-order norms of $u_*$.

For any fixed $\varepsilon>0$, we shall prove regularity and estimate of $u_*(\varepsilon)$.
Let $\delta = \varepsilon/(K+2)$ and $t_j = j\delta$ for $j = 0,\cdots,K+2$.
It is easy to show that for any $t\geq t_k$, $u_*(t) = e^{-(t-t_k)\nu A^s}u_k + w_k(t)$, where $u_k = u_*(t_k)$, and $w_k(t)$ solves the following Cauchy problem starting from $t= t_k$,
\begin{equation*}
\partial_t w_k + \nu A^s w_k = f(u_*,u_*),\quad w_k|_{t = t_k} = 0.
%\label{eqn: cauchy problem starting from t_k}
\end{equation*}
It admits an energy estimate, i.e., for all $t\geq 0$,
\begin{equation}
\begin{split}
&\;\|A^{(r_k-1+s)/2}w_k(t)\|_{L^2}^2 +\int_{t_k}^{t} \|A^{(r_k-1)/2+s}w_k\|_{L^2}^2\\
\leq &\; C\int_{t_k}^{t} \|A^{(r_k-1)/2}f(u_*,u_*)\|_{L^2}^2 \,d\tau\\
\leq &\; C\int_{t_k}^t \|u_*\|_{D(A)}^2(\|A^{s/2}u_*\|_{L^2}^2+\|A^{r_k/2}u_*\|_{L^2}^2) \,d\tau.
\end{split}
\label{eqn: energy estimate in the noncritical case for w_k crude form}
\end{equation}
We used \eqref{eqn: estimate of f_u in the noncritical case} in the last inequality.
Thanks to \eqref{eqn: estimate of the global solution} in Theorem \ref{thm: global existence and uniqueness} and the global $H^1$-estimate \eqref{eqn: L^2 energy estimate} of $u_*$, by the definition of $r_k$,
\begin{equation}
\begin{split}
&\;\|A^{(r_{k+1}-s)/2}w_k(t)\|_{L^2}^2 +\int_{t_k}^{t} \|A^{r_{k+1}/2}w_k\|_{L^2}^2\\
\leq &\; C\|u_0\|_{D(A)}^2\left(1+\int_{t_k}^t \|A^{r_k/2}u_*\|_{L^2}^2 \,d\tau\right),
\end{split}
\label{eqn: energy estimate in the noncritical case for w_k}
\end{equation}
where $C = C(\alpha,  s, n, \nu, \Omega, \|u_0\|_{D(A^{1/2})})$.
This holds for all $t\geq 0$.
For $k=0$, this together with \eqref{eqn: estimate of the global solution} implies
\begin{equation}
\|A^{(r_1-s)/2} w_0(t)\|_{L^2}^2 +\int_0^t \|A^{r_1/2}w_0\|_{L^2}^2\,d\tau\leq C\|u_0\|_{D(A)}^2.
\label{eqn: improved estimates of w_0}
\end{equation}
%where $C = C(\alpha,  s, n, \nu, \Omega, \|u_0\|_{D(A^{1/2})}, T)$.
For all $k\geq 1$, we write $u_*(\tau) = e^{-(\tau-t_{k-1})\nu A^s}u_{k-1}+ w_{k-1}(\tau)$ and derive from \eqref{eqn: energy estimate in the noncritical case for w_k} that
\begin{equation}
\begin{split}
&\;\|A^{(r_{k+1}-s)/2}w_k(t)\|_{L^2}^2 +\int_{t_k}^{t} \|A^{r_{k+1}/2}w_k\|_{L^2}^2\\
\leq &\; C\|u_0\|_{D(A)}^2\int_{t_k}^t \|A^{s/2}e^{-(\tau-t_k)\nu A^s} (A^{(r_k-s)/2}e^{-(t_{k}-t_{k-1})\nu A^s} u_{k-1})\|_{L^2}^2 \,d\tau\\
&\; +C\|u_0\|_{D(A)}^2\left(1+\int_{t_k}^t \|A^{r_k/2}w_{k-1}\|_{L^2}^2 \,d\tau\right)\\
\leq &\; C\|u_0\|_{D(A)}^2 \left(\|A^{(r_k-s)/2}e^{-\delta \nu A^s} u_{k-1}\|_{L^2}^2%\\
%&\; +C\|u_0\|_{D(A)}^2
+1+\int_{t_{k-1}}^t \|A^{r_k/2}w_{k-1}\|_{L^2}^2 \,d\tau\right)\\
\leq &\; C\|u_0\|_{D(A)}^2 \left[(\nu\delta)^{-\frac{2s-1}{s}}\|A^{(r_{k-1}-s)/2} u_{k-1}\|_{L^2}^2
+1+\int_{t_{k-1}}^t \|A^{r_k/2}w_{k-1}\|_{L^2}^2 \,d\tau\right].
\end{split}
%%\label{eqn: energy estimate in the noncritical case for w_k}
\label{eqn: improved estimates of w_k crude form}
\end{equation}
Here we used energy estimate for semigroup solutions.
%
%\begin{equation}
%\begin{split}
%&\;\|w_k(t)\|_{D(A^{(r_{k+1}-s)/2})}^2 + \int_{t_k}^t \|w_k\|_{D(A^{r_{k+1}/2})}^2\,d\tau\\
%\leq &\; C\int_{t_k}^t \|e^{-(\tau-t_k)\nu A^s}e^{-(t_k-t_{k-1})\nu A^s}u_{k-1}\|_{D(A^{r_k/2})}^2 + \|w_{k-1}\|_{D(A^{r_k/2})}^2+1 \,d\tau\\
%\leq &\; C\left[\|e^{-(t_k-t_{k-1})\nu A^s}u_{k-1}\|_{D(A^{(r_k-s)/2})}^2+\int_{t_{k-1}}^t\|w_{k-1}\|_{D(A^{r_k/2})}^2\,d\tau+ (t - t_k)\right]\\
%\leq &\; C\left[\left((t_k-t_{k-1})^{-\frac{2s-1}{s}}+1\right)\|u_{k-1}\|_{D(A^{(r_{k-1}-s)/2})}^2+\int_{t_{k-1}}^t\|w_{k-1}\|_{D(A^{r_k/2})}^2\,d\tau+ 1\right],
%\end{split}
%
%\end{equation}
%where $C = C(\alpha, \nu, s, K, n, \Omega, \|u_0\|_{D(A^{1/2})}, T)$.
Since $t\geq 0$ is arbitrary, this implies that
\begin{equation}
\begin{split}
&\;\|A^{(r_{k+1}-s)/2}w_k(t_{k+1})\|_{L^2}^2 +\int_{t_k}^{\infty} \|A^{r_{k+1}/2}w_k\|_{L^2}^2\\
\leq &\; C\|u_0\|_{D(A)}^2 (\nu\delta)^{-\frac{2s-1}{s}}\|A^{(r_{k-1}-s)/2} u_{k-1}\|_{L^2}^2\\
&\;+C\|u_0\|_{D(A)}^2 \left[
1+\int_{t_{k-1}}^\infty \|A^{r_k/2}w_{k-1}\|_{L^2}^2 \,d\tau\right].
\end{split}
%%\label{eqn: energy estimate in the noncritical case for w_k}
\label{eqn: improved estimates of w_k}
\end{equation}
Since $u_{k} = u_*(t_k) = e^{-\delta  v A^s}u_{k-1} + w_{k-1}(t_k)$ for all $k = 0,\cdots, K+1$,
\begin{equation}
\begin{split}
&\;\|A^{(r_k-s)/2} u_k\|_{L^2}^2\\
\leq &\; C\left(\|A^{(r_k-s)/2}e^{-\delta\nu A^s}u_{k-1}\|_{L^2}^2+\|A^{(r_k-s)/2}w_{k-1}(t_k)\|_{L^2}^2\right)\\
\leq &\;C\left[(\nu\delta)^{-\frac{2s-1}{s}}\|A^{(r_{k-1}-s)/2}u_{k-1}\|_{L^2}^2+\|A^{(r_k-s)/2}w_{k-1}(t_k)\|_{L^2}^2\right].
\end{split}
\label{eqn: bounding norm of u_k}
\end{equation}
%Combining this with \eqref{eqn: improved estimates of w_k}, we obtain that
%\begin{equation}
%\begin{split}
%&\;1+\|A^{(r_k-s)/2} u_k\|_{L^2}^2+\|A^{(r_{k+1}-s)/2}w_k(t_{k+1})\|_{L^2}^2 +\int_{t_k}^\infty \|A^{r_{k+1}/2}w_k\|_{L^2}^2\\
%\leq &\;C(\|u_0\|_{D(A)}^2 +1)((\nu\delta)^{-\frac{2s-1}{s}}+1)\\
%&\;\qquad \left[1+\|A^{(r_{k-1}-s)/2} u_{k-1}\|_{L^2}^2+ \|A^{(r_k-s)/2}w_{k-1}(t_k)\|_{L^2}^2+\int_{t_{k-1}}^t \|A^{r_k/2}w_{k-1}\|_{L^2}^2 \,d\tau\right].
%\end{split}
%\end{equation}
%We may assume the constants in \eqref{eqn: improved estimates of w_k} and \eqref{eqn: bounding norm of u_k} are all given by $C_1 = C_1(\alpha,  s, n, \nu, \Omega, \|u_0\|_{D(A^{1/2})})$.
To this end, we claim that for $j = 0,\cdots,K$,
%\begin{equation}
% \leq C((\nu\delta)^{-\frac{j(2s-1)}{s}}+1)(\|u_0\|^2_{D(A)}+1)^{j}\|u_0\|^2_{D(A)},
%\label{eqn: estimate of u_j}
%\end{equation}
%and
\begin{equation}
\begin{split}
&\;\|A^{(r_j-s)/2}u_j\|_{L^2}^2+\|A^{(r_{j+1}-s)/2}w_j(t_{j+1})\|_{L^2}^2+\int_{t_j}^\infty\|A^{r_{j+1}/2}w_j\|_{L^2}^2\,d\tau \\
\leq &\;C(\delta^{-\frac{j(2s-1)}{s}}+1)(\|u_0\|^2_{D(A)}+1)^{j}\|u_0\|^2_{D(A)},
\end{split}
\label{eqn: estimate of u_j and w_j}
\end{equation}
where $C = C(\alpha,  s, n, \nu, \Omega, \|u_0\|_{D(A^{1/2})},j)$.
We shall prove this by \eqref{eqn: improved estimates of w_k}, \eqref{eqn: bounding norm of u_k} and induction.
Indeed, for $j = 0$, \eqref{eqn: estimate of u_j and w_j} follows from \eqref{eqn: improved estimates of w_0} immediately.
Now suppose \eqref{eqn: estimate of u_j and w_j} holds for $j \leq k-1$ with some $k\geq 1$.
%Since $u_k = u_*(t_k) = e^{-(t_k-t_{k-1}) v A^s}u_{k-1} + w_{k-1}(t_k)$,
%\begin{equation}
%\begin{split}
%\|u_k\|_{D(A^{(r_k-s)/2})}^2\leq &\; C\left(\|e^{-\delta\nu A^s}u_{k-1}\|_{D(A^{(r_k-s)/2})}^2+\|w_{k-1}(t_k)\|_{D(A^{(r_k-s)/2})}^2\right)\\
%\leq &\;C(\delta^{-\frac{2s-1}{s}}+1)\|u_{k-1}\|_{D(A^{(r_{k-1}-s)/2})}^2+C\|w_{k-1}(t_k)\|_{D(A^{(r_k-s)/2})}^2\\
%\leq &\;C(\delta^{-\frac{k(2s-1)}{s}}+1)(\|u_0\|_{D(A)}^2+1),
%\end{split}
%\label{eqn: bounding norm of u_k}
%\end{equation}
Then for $j = k$, \eqref{eqn: estimate of u_j and w_j} follows immediately from \eqref{eqn: improved estimates of w_k} and \eqref{eqn: bounding norm of u_k}.
This justifies the claim.
%We let $t = t_{k+1}$ and $t = T$ in \eqref{eqn: improved estimates of w_k}, and obtain that
%\begin{equation*}
%\begin{split}
%&\;\|w_k(t_{k+1})\|_{D(A^{(r_{k+1}-s)/2})}^2 + \int_{t_k}^T \|w_k\|_{D(A^{r_{k+1}/2})}^2\,d\tau\\
%\leq &\; C\left[(\delta^{-\frac{(2s-1)}{s}}+1)\|u_{k-1}\|_{D(A^{(r_{k-1}-s)/2})}^2+\int_{t_{k-1}}^T\|w_{k-1}\|_{D(A^{r_k/2})}^2\,d\tau+ 1\right]\\
%\leq &\; C(\delta^{-\frac{k(2s-1)}{s}}+1)(\|u_0\|^2_{D(A)}+1).
%\end{split}
%\end{equation*}
%This proves \eqref{eqn: estimate of w_j} for the case $j = k$.
%By induction, \eqref{eqn: estimate of u_j} and \eqref{eqn: estimate of w_j} are established for $j = 0,\cdots, K$.

To this end, we shall perform the last-step improvement to derive an estimate for $u_*(\varepsilon) = e^{-\delta\nu A^s}u_{K+1}+w_{K+1}(t_{K+2})$.
Similar to \eqref{eqn: energy estimate in the noncritical case for w_k crude form} and \eqref{eqn: improved estimates of w_k crude form},
\begin{equation*}
\begin{split}
&\;\|A^{1+s/2}w_{K+1}(t_{K+2})\|_{L^2}^2\\
\leq &\; C\int_{t_{K+1}}^{\infty} \|Af(u_*,u_*)\|_{L^2}^2\,d\tau\\
\leq &\; C\int_{t_{K+1}}^{\infty} \|u_*\|_{D(A)}^2(\|A^{s/2}u_*\|_{L^2}^2+\|A^{3/2}u_*\|_{L^2}^2)\,d\tau\\
%\leq &\; C\int_{t_{K+1}}^{t_{K+2}} \|e^{-(\tau-t_K)\nu A^s}u_K+w_K\|_{D(A^{3/2})}^2+1 \,d\tau\\
\leq &\; C\|u_0\|_{D(A)}^2\left(1
+ \|A^{(3-s)/2}e^{-\delta\nu A^s}u_K\|_{L^2}^2
+\int_{t_{K+1}}^{\infty} \|A^{3/2}w_K\|_{L^2}^2\,d\tau
\right). %\\
%\leq &\; C\|u_0\|_{D(A)}^2\left(1
%+ \|A^{(3-s)/2}e^{-\delta\nu A^s}u_K\|_{L^2}^2
%+\int_{t_{K+1}}^{\infty} \|A^{3/2}w_K\|_{L^2}^2\,d\tau
%\right)
\end{split}
\end{equation*}
Note that $r_{K+1}\geq 3$.
By \eqref{eqn: estimate of u_j and w_j} with $j = K$,
\begin{equation*}
\begin{split}
&\;\|A^{1+s/2}w_{K+1}(t_{K+2})\|_{L^2}^2\\
\leq &\; C\|u_0\|_{D(A)}^2\left(1
+ (\nu\delta)^{-\frac{3-r_K}{s}}\|A^{(r_K-s)/2}u_K\|_{L^2}^2\right)\\
&\;+C\|u_0\|_{D(A)}^2\int_{t_{K+1}}^{\infty} \|A^{r_{K+1}/2}w_K\|_{L^2}^2\,d\tau\\
\leq &\; C\|u_0\|_{D(A)}^2(\delta^{-\frac{1-s}{s}}+1)(\|u_0\|^2_{D(A)}+1)^{K+1}.
\end{split}
%\label{eqn: estimate of w_K+1}
\end{equation*}
%
%
%\begin{equation*}
%\begin{split}
%&\;\|w_{K+1}(t_{K+2})\|_{D(A^{1+s/2})}^2\\
%\leq &\; C\int_{t_{K+1}}^{t_{K+2}} \|e^{-(\tau-t_K)\nu A^s}u_K\|_{D(A^{3/2})}^2+\|w_K\|_{D(A^{r_{K+1}/2})}^2+1 \,d\tau\\
%\leq &\; C\left[\|e^{-(t_{K+1}-t_K)\nu A^s}u_K\|_{D(A^{(3-s)/2})}^2+\int_{t_{K+1}}^{t_{K+2}}\|w_K\|_{D(A^{r_{K+1}/2})}^2\,d\tau+1\right]\\
%\leq &\; C\left[\left((t_{K+1}-t_K)^{-\frac{3-r_K}{s}}+1\right)\|u_K\|_{D(A^{(r_{K}-s)/2})}^2+C(\delta^{-\frac{K(2s-1)}{s}}+ 1)(\|u_0\|^2_{D(A)}+1)+1\right]\\
%\leq &\; C(\delta^{-\frac{1-s}{s}}+ 1)(\|u_0\|^2_{D(A)}+1).
%\end{split}
%\end{equation*}
%
%
On the other hand, since \eqref{eqn: bounding norm of u_k} also holds for $j = K+1$,
\begin{equation*}
\begin{split}
&\;\|A^{1+s/2}e^{-\delta \nu A^s}u_{K+1}\|_{L^2}^2\\
\leq &\;C(\nu\delta)^{-\frac{(2+s)-(r_{K+1}-s)}{s}}\|A^{(r_{K+1}-s)/2}u_{K+1}\|_{L^2}^2\\
\leq &\; C(\delta^{-1}+1)(\|u_0\|^2_{D(A)}+1)^K\|u_0\|^2_{D(A)}.
\end{split}
%\label{eqn: estimate of u_K}
\end{equation*}
%where $C = C(\alpha,  s, n, \nu, \Omega, \|u_0\|_{D(A^{1/2})},K)$.
%
Combining these two estimates, we obtain that
\begin{equation}
\begin{split}
&\;
\|A^{1+s/2}u_*(t_{K+2})\|_{L^2}^2\\
%\leq &\;C\left(\|e^{-\delta\nu A^s}u_{K+1}\|_{D(A^{1+s/2})}^2+\|w_{K+1}(t_{K+2})\|_{D(A^{1+s/2})}^2\right)\\
%\leq &\;C\left[\left(\delta^{-\frac{(2+s)-(r_{K+1}-s)}{s}}+1\right)\|u_{K+1}\|_{D(A^{(r_{K+1}-s)/2})}^2+(\delta^{-\frac{1-s}{s}}+ 1)(\|u_0\|^2_{D(A)}+1)\right]\\
\leq &\;C(\delta^{-1}+1)(\|u_0\|^2_{D(A)}+1)^{K+1}\|u_0\|^2_{D(A)}\\
\leq &\;C(\varepsilon^{-1}+1)(\|u_0\|^2_{D(A)}+1)^{K+1}\|u_0\|^2_{D(A)},
\end{split}
\label{eqn: estimate of u_K+2}
\end{equation}
where $C = C(\alpha,  s, n, \nu, \Omega, \|u_0\|_{D(A^{1/2})},K)$.
Note that $K$ essentially depends on $s$.
By interpolation between \eqref{eqn: estimate of u_K+2} and $\|u_*(\varepsilon)\|_{D(A)}^2\leq C\|u_0\|^2_{D(A)}$, we know that for all $r\in[0,s/2]$,
\begin{equation*}
\|u_*(\varepsilon)\|_{D(A^{1+r})}^2\leq C(\varepsilon^{-\frac{2r}{s}}+1)\|u_0\|^2_{D(A)},
\end{equation*}
where $C = C(\alpha, s, n,\nu, \Omega, \|u_0\|_{D(A)})$.
In particular, as $\|u_0\|_{D(A)}\to 0$, $C$ converges to a universal constant depending on $\alpha$, $s$, $n$, $\nu$ and $\Omega$.
Since $\varepsilon>0$ is arbitrary, \eqref{eqn: estimate of the improved regularity of u_*} is proved.
%By the definition of $w_{K-1}$, \eqref{eqn: estimate of w_K-1} implies that
%\begin{equation}
%\begin{split}
%\int_{\varepsilon}^T\|u_*\|_{D(A^{r_K/2})}^2\,d\tau\leq &\;C\int_{t_{K-1}}^T\|e^{-(\tau-t_{K-1})\nu A^s}u_{K-1}\|_{D(A^{r_K/2})}^2+\|w_{K-1}\|_{D(A^{r_K/2})}^2\,d\tau\\
%\leq &\;C\|u_{K-1}\|_{D(A^{(r_K-s)/2})}^2\\
%&\;+C(\varepsilon^{-1} K)^{\frac{(K-1)(2s-1)}{s}}\|u_0\|^2_{D(A)} + C(\|u_0\|^2_{D(A)}+1)\\
%\leq &\;C(\varepsilon^{-1} K)^{\frac{(K-1)(2s-1)}{s}}\|u_0\|^2_{D(A)} + C(\|u_0\|^2_{D(A)}+1).
%\end{split}
%\end{equation}
%Hence, for all $r\in\left[1,1+\frac{2s-1}{2}[\frac{1-s}{2s-1}]\right]$, and all $t\in(0,T]$,
%\begin{equation}
%\begin{split}
%\int_t^T\|u_*\|_{D(A^{r_K/2})}^2\,d\tau\leq C(\varepsilon^{-1} K)^{\frac{(K-1)(2s-1)}{s}}\|u_0\|^2_{D(A)} + C(\|u_0\|^2_{D(A)}+1).
%\end{split}
%\end{equation}
\end{proof}
\end{theorem}

\subsection{Critical case: $s = 1/2$}
Now we consider the case $(n,s) = (2,1/2)$.
It is called critical since no easy bootstrapping argument  can be applied as before. %to show higher regularity of $u_*$.
%More sophisticated analysis is involved.
In what follows, we shall prove, in the fashion of re-constructing the solution, that $u_*$ has local H\"{o}lder continuity in time away from $t=0$ as a function valued in $D(A^{1+s/2})$; while the H\"{o}lder norm admits a singularity at $t = 0$ with certain growth rate as $t\rightarrow 0^+$.
This idea comes from the earlier studies of regularity of $L^p$-solution of the Navier-Stokes equation and semilinear parabolic equations \cite{fujita1964navier,giga1985solutions,giga1986solutions}.
To be more precise, we introduce the following definition.

\begin{definition}\label{def: define a more refined set of functions}
Fix $T\in(0,1]$ and let $w \in C_{[0,T]}(D(A))\cap L^2_T (D(A^{1+s/2}))$.
With $R>0$ and $\beta\in (0,1/2)$, we say $w\in B_{R,T}^\beta$ if and only if
\begin{enumerate}
\item $\|w(t)\|_{D(A)}\leq R$ for all $t\in [0,T]$;
\item $\|A^{s/2}w(t)\|_{D(A)}\leq R t^{-1/2}$ for all $t\in (0,T]$;
\item For all $0< t\leq t+h\leq T$,
\begin{align}
\|w(t+h)-w(t)\|_{D(A)}\leq &\;h^\beta t^{-\beta}R,\label{eqn: holder continuity of D(A) norm}\\
\|A^{s/2}(w(t+h)-w(t))\|_{D(A)}\leq &\;h^\beta t^{-(\beta+1/2)}R.\label{eqn: holder continuity of D(A^1+s/2) norm}
\end{align}
\end{enumerate}
\end{definition}

In fact, homogeneous solutions given by the semigroup $\{e^{-t\nu A^s}\}_{t\geq 0}$ is in this type of sets.
\begin{lemma}\label{lemma: semigroup solution is in that class}
For all $w_0\in D(A)$, $w(t) = e^{-t\nu A^s}w_0\in B_{R,T}^\beta$ with $R = C(\nu,\beta)\|w_0\|_{D(A)}$.
\begin{proof}
It is trivial that $\|w(t)\|_{D(A)}\leq \|w_0\|_{D(A)}$ and
\begin{equation*}
\begin{split}
&\;\|A^{s/2}w(t)\|_{D(A)}\\
\leq &\;\|e^{-t\nu A^s}A^{s/2}w_0\|_{L^2}+\|A^{s/2}e^{-t\nu A^s}Aw_0\|_{L^2}\\
\leq &\;\|A^{s/2}w_0\|_{L^2}+C(\nu) t^{-1/2}\|Aw_0\|_{L^2}\\
\leq &\;C(\nu) t^{-1/2}\|w_0\|_{D(A)}.
\end{split}
\end{equation*}
In the last inequality, we used the fact that $t \leq 1$.

To prove \eqref{eqn: holder continuity of D(A) norm}, we derive that
\begin{equation*}
\begin{split}
&\;\|e^{-(t+h)\nu A^s}w_0 - e^{-t\nu A^s}w_0\|_{D(A)}\\
= &\;\left\|\int_t^{t+h} \nu A^s e^{-\tau\nu A^s}w_0\,d\tau\right\|_{D(A)}
\leq C\int_t^{t+h} \tau^{-1}\,d\tau\cdot \|w_0\|_{D(A)}\\
\leq &\; C\ln\left(1+\frac{h}{t}\right)\|w_0\|_{D(A)}
\leq C(\beta)h^\beta t^{-\beta}\|w_0\|_{D(A)}.
\end{split}
\end{equation*}
Similarly,
\begin{equation*}
\begin{split}
&\;\|A^{s/2}(e^{-(t+h)\nu A^s}w_0 - e^{-t\nu A^s}w_0)\|_{D(A)}\\
= &\;\left\|\int_t^{t+h} \nu A^{3s/2} e^{-\tau\nu A^s}w_0\,d\tau\right\|_{D(A)}
\leq  C(\nu)\int_t^{t+h} \tau^{-3/2}\,d\tau\cdot \|w_0\|_{D(A)}\\
\leq &\; C(\nu)t^{-1/2}\left(1-\sqrt{\frac{t}{t+h}}\right)\|w_0\|_{D(A)}
\leq C(\nu,\beta)\frac{h^\beta}{t^{\beta+1/2}}\|w_0\|_{D(A)}.
\end{split}
\end{equation*}
In the last inequality, we used the fact that $1-x^{-1/2}\leq (x-1)^\beta$ for all $x\geq 1$.
Indeed, it is trivial when $x\geq 2$; for $x\in [1,2]$, $1-x^{-1/2}\leq 1-x^{-1}\leq x-1 \leq (x-1)^\beta$.
This completes the proof.
\end{proof}
\end{lemma}

The following lemma is the key to show existence of the solution in the type of sets $B_{R,T}^\beta$.
%It plays a similar role of \eqref{eqn: energy estimate of a solution starting from 0} in constructing a solution, but the characterization is much more refined.

\begin{lemma}\label{lemma: estimate of nonlinear term in the critical case}
Fix $T\in(0,1]$ and $\beta\in (0,1/2)$.
For $w_1\in B_{R_1,T}^\beta$ and $w_2\in B_{R_2,T}^\beta$, let
\begin{equation*}
v[w_1,w_2](t) = \int_0^t e^{-(t-\tau)\nu A^s}f(w_1,w_2)\,d\tau,\quad t\in[0,T].
\end{equation*}
Then $v[w_1,w_2](t)\in B_{CR_1R_2,T}^\beta$, where $C = C(\alpha, \nu, \Omega, \beta)$.
\begin{proof}
%Before we check the definition of $B_{CR_1R_2,T}^\beta$,
It is helpful to first derive some estimates for $f(w_1, w_2)$.
%We apply
By Lemma \ref{lemma: unified spatial estimate of f} with $r = 2-s$, % and $r' = r$ to find that
\begin{equation}
%\begin{split}
\|f(w_1,w_2)(t)\|_{D(A^{1-s/2})} %\leq %&\; C\|w_1\|_{D(A)}(\|A^{1/2} w_2\|_{L^2}+\|A^{(r+1)/2} w_2\|_{L^2})\\
\leq C\|w_1\|_{D(A)}\|A^{s/2}w_2\|_{D(A)}\leq CR_1 R_2 t^{-1/2},
\label{eqn: norm of f used in the critical case}
%\end{split}
\end{equation}
where $C = C(\alpha, \Omega)$.
In addition, %t is also useful to derive that
for all $t\geq \tau>0$,
\begin{equation}
\begin{split}
&\;\|A^{1-s/2}(f(w_1,w_2)(\tau) - f(w_1,w_2)(t))\|_{L^2}\\
\leq &\;\|f(w_1(\tau)-w_1(t),w_2(\tau))\|_{D(A^{1-s/2})}+\|f(w_1(t),w_2(\tau)-w_2(t))\|_{D(A^{1-s/2})}\\
\leq &\;C\|w_1(\tau)-w_1(t)\|_{D(A)}\|A^{s/2}w_2(\tau)\|_{D(A)}\\
&\;+C\|w_1(t)\|_{D(A)}\|A^{s/2}(w_2(\tau)-w_2(t))\|_{D(A)}\\
%\leq &\;C\frac{(t-\tau)^\beta}{\tau^\beta} R_1 \tau^{-1/2}R_2 + CR_1 \frac{(t-\tau)^\beta}{\tau^{\beta+1/2}}R_2\\
\leq &\;C\frac{(t-\tau)^\beta}{\tau^{\beta+1/2}}R_1R_2,
\end{split}
\label{eqn: time Holder norm of f used in the critical case}
\end{equation}
where $C= C(\alpha, \Omega)$.

For brevity, we write $v[w_1,w_2](t)$ as $v(t)$ in the following.
\setcounter{step}{0}
\begin{step}
We start from $\|v(t)\|_{D(A)}$.
Thanks to \eqref{eqn: norm of f used in the critical case},
\begin{equation*}
\begin{split}
\|v(t)\|_{D(A)} \leq &\;\int_0^t (\|e^{-(t-\tau)\nu A^s}\|_{\mathcal{L}(L^2)}+\|A^{s/2}e^{-(t-\tau)\nu A^s}\|_{\mathcal{L}(L^2)})\|f(w_1,w_2)\|_{D(A^{1-s/2})}\,d\tau\\
\leq &\;C\int_0^t (1+(t-\tau)^{-1/2})R_1R_2 \tau^{-1/2}\,d\tau\leq CR_1 R_2,
\end{split}
\end{equation*}
where $C = C(\alpha, \nu, \Omega)$ and $\|\cdot\|_{\mathcal{L}(L^2)}$ denotes the operator norm from $L^2(\Omega)$ to itself.
Here we use the fact that $(t-\tau)<T\leq 1$.
\end{step}

\begin{step}
We make estimate for $\|A^{s/2}v(t)\|_{D(A)}$.
Thanks to \eqref{eqn: norm of f used in the critical case} and \eqref{eqn: time Holder norm of f used in the critical case},
\begin{equation*}
\begin{split}
&\;\|A^{s/2}v(t)\|_{D(A)}\\
\leq &\;\|A^{s/2}v(t)\|_{L^2}+\left\|\int_0^t A^s e^{-(t-\tau)\nu A^s} A^{1-s/2}f(w_1,w_2)(t)\,d\tau\right\|_{L^2}\\
&\;+ \left\|\int_0^t A^s e^{-(t-\tau)\nu A^s} [A^{1-s/2}f(w_1,w_2)(\tau) - A^{1-s/2}f(w_1,w_2)(t)]\,d\tau\right\|_{L^2}\\
\leq &\; C\|v(t)\|_{D(A)}+\nu^{-1}\left\|\int_0^t \left.\frac{d}{d\eta}\right|_{\eta = \tau} \left(e^{-(t-\eta)\nu A^s}\right) A^{1-s/2}f(w_1,w_2)(t)\,d\tau\right\|_{L^2}\\
&\;+\int_0^t \left\|A^s e^{-(t-\tau)\nu A^s}\right\|_{\mathcal{L}(L^2)}\|A^{1-s/2}(f(w_1,w_2)(\tau) - f(w_1,w_2)(t))\|_{L^2}\,d\tau\\
\leq &\; CR_1R_2+\nu^{-1}\|(Id - e^{-t\nu A^s}) A^{1-s/2}f(w_1,w_2)(t)\|_{L^2}\\
&\;+C\nu^{-1}\int_0^t (t-\tau)^{-1}\frac{(t-\tau)^\beta}{\tau^{\beta+1/2}}R_1R_2\,d\tau\\
%\leq &\; Ct^{-1/2}R_1R_2+C\|w_1(t)\|_{D(A)}\|A^{s/2}w_2(t)\|_{D(A)}\\
\leq &\; Ct^{-1/2}R_1R_2,
\end{split}
\end{equation*}
where $C = C(\alpha, \nu, \Omega, \beta)$.
\end{step}

\begin{step}\label{step: time holder continuity of D(A) norm}
We check \eqref{eqn: holder continuity of D(A) norm} for $v$.
For all $0<t\leq t+h\leq T$, by \eqref{eqn: norm of f used in the critical case} and \eqref{eqn: time Holder norm of f used in the critical case},
\begin{equation}
\begin{split}
&\;\|v(t+h)-v(t)\|_{D(A)}\\
=&\;\left\|\int_0^{t+h} e^{-(t+h-\tau)\nu A^s}f(w_1,w_2)(\tau)\,d\tau-\int_0^t e^{-(t-\tau)\nu A^s}f(w_1,w_2)(\tau)\,d\tau\right\|_{D(A)}\\
\leq &\;\left\|\int_0^h e^{-(t+h-\tau)\nu A^s} f(w_1,w_2)(\tau)\,d\tau\right\|_{D(A)}\\
&\;+\left\|\int_0^t e^{-(t-\tau)\nu A^s} (f(w_1,w_2)(\tau+h)-f(w_1,w_2)(\tau))\,d\tau\right\|_{D(A)}\\
\leq &\;\int_0^h \|f(w_1,w_2)(\tau)\|_{L^2}\,d\tau\\
&\;+\int_0^h \|A^{s/2}e^{-(t+h-\tau)\nu A^s}\|_{\mathcal{L}(L^2)} \|A^{1-s/2}f(w_1,w_2)(\tau)\|_{L^2}\,d\tau\\
&\;+\int_0^t \|(f(w_1,w_2)(\tau+h)-f(w_1,w_2)(\tau))\|_{L^2}\,d\tau\\
&\;+\int_0^t \|A^{s/2}e^{-(t-\tau)\nu A^s}\|_{\mathcal{L}(L^2)} \|A^{1-s/2}(f(w_1,w_2)(\tau+h)-f(w_1,w_2)(\tau))\|_{L^2}\,d\tau\\
%\leq &\;C\int_0^h (t+h-\tau)^{-1/2}\|f(w_1,w_2)(\tau)\|_{D(A^{1-s/2})}\,d\tau\\
%&\;+C\int_0^t (t-\tau)^{-1/2}\|f(w_1,w_2)(\tau+h)-f(w_1,w_2)(\tau)\|_{D(A^{1-s/2})}\,d\tau\\
\leq &\;C\int_0^h (t+h-\tau)^{-1/2}R_1R_2 \tau^{-1/2}\,d\tau+C\int_0^t (t-\tau)^{-1/2}R_1R_2 \frac{h^\beta}{\tau^{\beta+1/2}}\,d\tau\\
\leq &\;CR_1R_2\int_0^h (t+h-\tau)^{-1/2} \tau^{-1/2}\,d\tau+Ch^\beta t^{-\beta} R_1R_2,
\end{split}
\label{eqn: derivation in step 3 in the critical case}
\end{equation}
where $C = C(\alpha, \nu, \Omega, \beta)$.
If $t\geq h$,
\begin{equation*}
\int_0^h (t+h-\tau)^{-1/2} \tau^{-1/2}\,d\tau\leq \int_0^h t^{-1/2} \tau^{-1/2}\,d\tau \leq Ch^{1/2}t^{-1/2} \leq Ch^\beta t^{-\beta}.
\end{equation*}
Otherwise, if $t < h$,
\begin{equation*}
\int_0^h (t+h-\tau)^{-1/2} \tau^{-1/2}\,d\tau\leq \int_0^{t+h} (t+h-\tau)^{-1/2} \tau^{-1/2}\,d\tau\leq C\leq Ch^\beta t^{-\beta}.
\end{equation*}
Combining the above estimates with \eqref{eqn: derivation in step 3 in the critical case}, we find that
\begin{equation*}
\|v(t+h)-v(t)\|_{D(A)}\leq Ch^\beta t^{-\beta} R_1R_2,
\end{equation*}
which is \eqref{eqn: holder continuity of D(A) norm}.

\end{step}

\begin{step}
We check \eqref{eqn: holder continuity of D(A^1+s/2) norm} for $v$.
Consider $\|v(t+h)-v(t)\|_{D(A^{1+s/2})}$ with $0<t\leq t+h\leq T$.
First we assume that $h\leq t/2$.
It is known that $\|A^{s/2}(v(t+h)-v(t))\|_{D(A)}\leq \|A^{s/2}(v(t+h)-v(t))\|_{L^2}+\|A^{1+s/2}(v(t+h)-v(t))\|_{L^2}$.
We focus on the second term as the first term can be handled using Step \ref{step: time holder continuity of D(A) norm}.
We calculate that
\begin{equation}
\begin{split}
&\;A^{1+s/2}(v(t+h)-v(t))\\
=&\;\int_0^{t+h}A^s e^{-(t+h-\tau)\nu A^s}A^{1-s/2}f(w_1,w_2)(\tau)\,d\tau\\
&\;-\int_0^{t}A^s e^{-(t-\tau)\nu A^s}A^{1-s/2}f(w_1,w_2)(\tau)\,d\tau\\
=
%
%&\;\int_0^{t-h}A^s e^{-(t+h-\tau)\nu A^s}A^{1-s/2}(f(w_1,w_2)(\tau)-f(w_1,w_2)(t))\,d\tau\\
%&\;-\int_0^{t-h}A^s e^{-(t-\tau)\nu A^s}A^{1-s/2}(f(w_1,w_2)(\tau)-f(w_1,w_2)(t))\,d\tau\\
%
&\;\int_0^{t-h}A^s (e^{-h\nu A^s}-Id)e^{-(t-\tau)\nu A^s} A^{1-s/2}(f(w_1,w_2)(\tau)-f(w_1,w_2)(t))\,d\tau\\
&\;+\int_{t-h}^{t+h}A^s e^{-(t+h-\tau)\nu A^s}A^{1-s/2}(f(w_1,w_2)(\tau)-f(w_1,w_2)(t+h))\,d\tau\\
&\;-\int_{t-h}^{t}A^s e^{-(t-\tau)\nu A^s}A^{1-s/2}(f(w_1,w_2)(\tau)-f(w_1,w_2)(t))\,d\tau\\
&\;+\int_0^{t-h}A^s e^{-(t+h-\tau)\nu A^s}A^{1-s/2}f(w_1,w_2)(t)\,d\tau\\
&\;+\int_{t-h}^{t+h}A^s e^{-(t+h-\tau)\nu A^s}A^{1-s/2}f(w_1,w_2)(t+h)\,d\tau\\
&\;-\int_0^t A^s e^{-(t-\tau)\nu A^s}A^{1-s/2}f(w_1,w_2)(t)\,d\tau\\
%=&\;\int_0^{t-h}A^s (e^{-h\nu A^s}-Id)e^{-(t-\tau)\nu A^s} A^{1-s/2}(f(w_1,w_2)(\tau)-f(w_1,w_2)(t))\,d\tau\\
%&\;+\int_{t-h}^{t+h}A^s e^{-(t+h-\tau)\nu A^s}A^{1-s/2}(f(w_1,w_2)(\tau)-f(w_1,w_2)(t+h))\,d\tau\\
%&\;-\int_{t-h}^{t}A^s e^{-(t-\tau)\nu A^s}A^{1-s/2}(f(w_1,w_2)(\tau)-f(w_1,w_2)(t))\,d\tau\\
%&\;+\nu^{-1}(e^{-2h\nu A^s}-e^{-(t+h)\nu A^s})A^{1-s/2}f(w_1,w_2)(t)\\
%&\;+\nu^{-1}(Id-e^{-2h\nu A^s})A^{1-s/2}f(w_1,w_2)(t+h)\\
%&\;-\nu^{-1}(Id-e^{-t\nu A^s})A^{1-s/2}f(w_1,w_2)(t)\\
=: &\; I_1+I_2+I_3+I_4+I_5+I_6.
\end{split}
\label{eqn: split time difference into six terms}
\end{equation}

Take $\beta'= (\beta+\frac{1}{2})/2 \in(\beta,1/2)$.
\begin{equation*}
\begin{split}
\|I_1\|_{L^2}
%=&\;\left\|\int_0^{t-h}A^{-\beta's} (e^{-h\nu A^s}-Id)A^{(1+\beta')s}e^{-(t-\tau)\nu A^s} A^{1-s/2}(f(w_1,w_2)(\tau)-f(w_1,w_2)(t))\,d\tau\right\|_{L^2}\\
\leq &\;\int_0^{t-h}\|A^{-\beta's} (e^{-h\nu A^s}-Id)\|_{\mathcal{L}(L^2)}\|A^{(1+\beta')s}e^{-(t-\tau)\nu A^s}\|_{\mathcal{L}(L^2)} \\
&\;\qquad\qquad\cdot \|A^{1-s/2}(f(w_1,w_2)(\tau)-f(w_1,w_2)(t))\|_{L^2}\,d\tau.
\end{split}
\end{equation*}
Since $\|A^{-\beta's} (e^{-h\nu A^s}-Id)\|_{\mathcal{L}(L^2)}\leq C(\nu,\beta')h^{\beta'}$ \cite{fujita1964navier} and $h\leq t/2$ by assumption,
\begin{equation}
\begin{split}
\|I_1\|_{L^2}\leq &\;C\int_0^{t-h}h^{\beta'}(t-\tau)^{-(1+\beta')}\frac{(t-\tau)^\beta}{\tau^{\beta+1/2}}R_1R_2\,d\tau\\
= &\;Ch^{\beta'}R_1R_2\int_0^{t/2}+\int_{t/2}^{t-h}(t-\tau)^{-(1+\beta'-\beta)}\tau^{-(\beta+1/2)}\,d\tau\\
%\leq &\;Ch^{\beta'}R_1R_2\left(t^{-(1+\beta'-\beta)}\int_0^{t/2}\tau^{-(\beta+1/2)}\,d\tau+\int_{t/2}^{t-h}(t-\tau)^{-(1+\beta'-\beta)}t^{-(\beta+1/2)}\,d\tau\right)\\
%\leq &\;Ch^{\beta'}R_1R_2\left(t^{-(1/2+\beta')}+h^{-(\beta'-\beta)}t^{-(\beta+1/2)}\right)\\
\leq &\;CR_1R_2h^{\beta}t^{-(\beta+1/2)},
\end{split}
\label{eqn: bounding I_1}
\end{equation}
where $C = C(\alpha, \nu, \Omega, \beta)$.
In the last step, we used the fact that $h\leq t$ and $\beta'\geq\beta$.

By \eqref{eqn: time Holder norm of f used in the critical case},
\begin{equation}
%\begin{split}
\|I_2\|_{L^2} %&\;\int_{t-h}^{t+h}(t+h-\tau)^{-1}\|A^{1-s/2}(f(w_1,w_2)(\tau)-f(w_1,w_2)(t+h))\|_{L^2}\,d\tau\\
\leq  C\int_{t-h}^{t+h} (t+h-\tau)^{-1}\cdot \frac{(t+h-\tau)^{\beta}}{\tau^{-(\beta+1/2)}}R_1R_2 \leq CR_1R_2h^{\beta}t^{\beta+1/2},
%\end{split}
\label{eqn: bounding I_2}
\end{equation}
and similarly,
\begin{equation}
%\begin{split}
\|I_3\|_{L^2} %\leq &\;\int_{t-h}^{t}\|A^s e^{-(t-\tau)\nu A^s}\|_{\mathcal{L}(L^2)}\|A^{1-s/2}(f(w_1,w_2)(\tau)-f(w_1,w_2)(t))\|_{L^2}\,d\tau\\
\leq C\int_{t-h}^{t} (t-\tau)^{-1}\frac{(t-\tau)^\beta}{\tau^{\beta+1/2}}R_1R_2 \leq CR_1R_2h^{\beta}t^{-(\beta+1/2)},
%\end{split}
\label{eqn: bounding I_3}
\end{equation}
where $C = C(\alpha, \nu, \Omega, \beta)$.

%
%\begin{equation}
%\begin{split}
%&\;\int_0^{t-h}A^s (e^{-h\nu A^s}-Id)e^{-(t-\tau)\nu A^s} A^{1-s/2}(f(w_1,w_2)(\tau)-f(w_1,w_2)(t))\,d\tau\\
%&\;+\int_{t-h}^{t+h}A^s e^{-(t+h-\tau)\nu A^s}A^{1-s/2}(f(w_1,w_2)(\tau)-f(w_1,w_2)(t+h))\,d\tau\\
%&\;-\int_{t-h}^{t}A^s e^{-(t-\tau)\nu A^s}A^{1-s/2}(f(w_1,w_2)(\tau)-f(w_1,w_2)(t))\,d\tau\\
%&\;+\nu^{-1}(e^{-2h\nu A^s}-e^{-(t+h)\nu A^s})A^{1-s/2}f(w_1,w_2)(t)\\
%&\;+\nu^{-1}(Id-e^{-2h\nu A^s})A^{1-s/2}f(w_1,w_2)(t+h)\\
%&\;-\nu^{-1}(Id-e^{-t\nu A^s})A^{1-s/2}f(w_1,w_2)(t)\\
%=: &\; I_1+I_2+I_3+I_4+I_5+I_6.
%\end{split}
%\label{eqn: split time difference into six terms}
%\end{equation}
%

The rest of the terms in \eqref{eqn: split time difference into six terms} can be handled as follows.
\begin{equation*}
\begin{split}
&\;\nu(I_4+I_5+I_6)\\
= &\;(e^{-2h\nu A^s}-e^{-(t+h)\nu A^s})A^{1-s/2}f(w_1,w_2)(t)\\
&\;+(Id-e^{-2h\nu A^s})A^{1-s/2}(f(w_1,w_2)(t+h)-f(w_1,w_2)(t))\\
&\;+(Id-e^{-2h\nu A^s})A^{1-s/2}f(w_1,w_2)(t)\\
&\;-(Id-e^{-t\nu A^s})A^{1-s/2}f(w_1,w_2)(t)\\
%= &\;(e^{-t\nu A^s}-e^{-(t+h)\nu A^s})A^{1-s/2}f(w_1,w_2)(t)\\
%&\;+(Id-e^{-2h\nu A^s})A^{1-s/2}(f(w_1,w_2)(t+h)-f(w_1,w_2)(t))\\
= &\;A^{-\beta s}(Id - e^{-h\nu A^s})A^{\beta s}e^{-t\nu A^s}A^{1-s/2}f(w_1,w_2)(t)\\
&\;+(Id-e^{-2h\nu A^s})A^{1-s/2}(f(w_1,w_2)(t+h)-f(w_1,w_2)(t)).
\end{split}
\end{equation*}
By virtue of \eqref{eqn: norm of f used in the critical case} and \eqref{eqn: time Holder norm of f used in the critical case},
\begin{equation}
\begin{split}
&\;\|I_4+I_5+I_6\|_{L^2}\\
\leq &\;C\|A^{-\beta s}(Id - e^{-h\nu A^s})\|_{\mathcal{L}(L^2)}\|A^{\beta s}e^{-t\nu A^s}\|_{\mathcal{L}(L^2)}\|A^{1-s/2}f(w_1,w_2)(t)\|_{L^2}\\
&\;+C\|(Id-e^{-2h\nu A^s})\|_{\mathcal{L}(L^2)}\|A^{1-s/2}(f(w_1,w_2)(t+h)-f(w_1,w_2)(t))\|_{L^2}\\
\leq &\;Ch^\beta t^{-(\beta+1/2)}R_1R_2,
\end{split}
\label{eqn: bounding I_456}
\end{equation}
where $C = C(\alpha, \nu, \Omega, \beta)$.
Combining \eqref{eqn: bounding I_1}-%, \eqref{eqn: bounding I_2}, \eqref{eqn: bounding I_3},
\eqref{eqn: bounding I_456}, and
\begin{equation*}
\|A^{s/2}(v(t+h)-v(t))\|_{L^2} \leq Ch^\beta t^{-\beta} R_1R_2 \leq Ch^\beta t^{-(\beta+1/2)} R_1R_2,
\end{equation*}
we establish \eqref{eqn: holder continuity of D(A^1+s/2) norm} for $v$ provided that $h\leq t/2$.

If $h>t/2$, there exist $N\in \mathbb{N}_+$ and $1+\kappa \in (\sqrt{3/2},3/2]$, such that $t+h = t(1+\kappa)^N$.
In fact, it suffices to consider $N = 1,2,2^2,\cdots$, and there will be exactly one such $N$ satisfying the above condition; $\kappa$ will follow from the choice of $N$.
With abuse of notations, let $t_j = t(1+\kappa)^j$ for $j = 0,\cdots, N$.
Then by \eqref{eqn: holder continuity of D(A^1+s/2) norm} for the case $h\leq t/2$,
%\begin{equation*}
%\begin{split}
%&\;\|A^{s/2}(v(t_j) - v(t_{j-1}))\|_{D(A)} \\
%\leq &\;CR_1R_2 (t_j-t_{j-1})^\beta t_{j-1}^{-(\beta+1/2)}
%\leq CR_1R_2 \kappa^\beta (t_0(1+\kappa)^{j-1})^{-1/2}.
%\end{split}
%\end{equation*}
%Hence,
\begin{equation*}
\begin{split}
&\;\|A^{s/2}(v(t_N) - v(t_0))\|_{D(A)} \\
\leq &\;\sum_{j = 1}^N \|A^{s/2}(v(t_j) - v(t_{j-1}))\|_{D(A)}\\
\leq &\;CR_1R_2 \sum_{j=1}^N(t_j-t_{j-1})^\beta t_{j-1}^{-(\beta+1/2)}\\
\leq &\;CR_1R_2 t^{-1/2}\kappa^\beta \sum_{j = 1}^N(1+\kappa)^{-(j-1)/2}\\
\leq &\;%CR_1R_2 t^{-1/2}
%\leq
CR_1R_2 h^\beta t^{-(\beta+1/2)},
\end{split}
\end{equation*}
where $C = C(\alpha, \nu, \Omega, \beta)$.
The last inequality follows from $h\geq t/2$.
\end{step}

This completes the proof.
\end{proof}
\end{lemma}

With Lemma \ref{lemma: estimate of nonlinear term in the critical case}, we have the following result in the critical case which re-constructs the solution obtained in Proposition \ref{prop: local existence and uniqueness}, yet with refined characterization of its regularity.
However, we shall additionally need the initial data to be small.
\begin{proposition}[Local well-posedness in a refined class]\label{prop: local existence and uniqueness in critical case}
Assume \eqref{eqn: assumption on the regularity of the domain} and let $(n,s) = (2,1/2)$.
For given $\beta\in (0,1/2)$, there exists an $\varepsilon = \varepsilon(\alpha, \nu, \Omega, \beta)\in (0,1]$, such that if $u_0\in D(A)$ with $\|u_0\|_{D(A)}\leq \varepsilon$, then the unique local solution $u$ obtained in Proposition \ref{prop: local existence and uniqueness} satisfies
%$u(x,t)\in C_{[0,1]}(D(A))\cap L^2_{[0,1]}(D(A^{1+s/2}))$, solving \eqref{eqn: equivalent formulation of LANS equation} with initial condition $u|_{t = 0} = u_0$ (and boundary conditions \eqref{eqn: boundary conditions of LANS fractional} if $\Omega$ is a smooth bounded domain).
%It satisfies \eqref{eqn: estimate of the local solution} and
$u\in B^\beta_{C\|u_0\|_{D(A)}, 1}$ with $C = C(\alpha,\nu,\Omega,\beta)$.
\begin{proof}
Instead of proving the regularity of the local solution $u$ directly, the proof uses a fixed-point iteration to re-construct the solution in the refined class.

Let $M = \|u_0\|_{D(A)}<+\infty$.
Take $T=1$ and fix $\beta\in(0,1/2)$, we denote
\begin{equation*}
\begin{split}
B:= \left\{u\in C_T(D(A))\right.&\cap L^2_T(D(A^{1+s/2})): u|_{t = 0} = u_0,\;u-e^{-t\nu A^s}u_0 \in B_{M,T}^\beta,\\
&\;\left.\|u-e^{-t\nu A^s}u_0\|_{L^\infty_T(D(A))\cap L^2_T(D(A^{1+s/2}))}\leq M\right\}.
\end{split}
\end{equation*}
$B$ is nonempty (since $e^{-t\nu A^s}u_0\in B$) and closed in $C_T(D(A))\cap L^2_T(D(A^{1+s/2}))$.
%By Lemma \ref{lemma: semigroup solution is in that class},
By definition, it holds for all $u\in B$ that
\begin{equation}
\|u\|_{L^\infty_T(D(A))\cap L^2_T(D(A^{1+s/2}))}\leq C(\nu)M.
\label{eqn: estimate for u in B}
\end{equation}

Consider the map $Q:\,u\mapsto Qu := w$, where $u\in B$ and $w$ solves
\begin{equation*}
\partial_t w+\nu A^s w = f(u,u),\quad w|_{t= 0} = u_0.
\end{equation*}
%
%
%in the proof of Proposition \ref{prop: local existence and uniqueness}, but restricted on $B'$.
%That is
Since $f(u,u)\in L^2_T(D(A^{1-s/2}))$ thanks to \eqref{eqn: norm of f used in the critical case}, the existence and uniqueness of $w\in L^\infty_T(D(A))\cap L^2_T(D(A^{1+s/2}))$ can be established easily, e.g., by Galerkin approximation.

We claim that $Q$ is well-defined from $B$ to itself if $M\ll 1$, with smallness of $M$ depending on $\alpha$, $\nu$, $\Omega$ and $\beta$.
Firstly, $\tilde{w} : = Qu- e^{-t\nu A^s}u_0 $ solves
\begin{equation*}
\partial_t \tilde{w}+\nu A^s \tilde{w} = f(u,u),\quad \tilde{w}|_{t = 0}  = 0.
\end{equation*}
In the view of energy estimate, \eqref{eqn: norm of f used in the critical case} and \eqref{eqn: estimate for u in B},
\begin{equation}
\|\tilde{w}\|_{L^\infty_T(D(A))\cap L^2_T(D(A^{1+s/2}))}\leq C(\nu)\|f(u,u)\|_{L^2_T D(A^{1-s/2})}\leq C_0M^2,
\label{eqn: estimate of the deviation from the semigroup homogeneous solution}
\end{equation}
where $C_0 = C_0(\alpha, \nu,\Omega)$.
If $M\leq C_0^{-1}$, $Qu$ satisfies
\begin{equation*}
\|Qu-e^{-t\nu A^s}u_0\|_{L^\infty_T(D(A))\cap L^2_T(D(A^{1+s/2}))}\leq M.
\end{equation*}
In addition, we may write
\begin{equation*}
Qu(t) = e^{-t\nu A^s}u_0 + \int_0^t e^{-(t-\tau)\nu A^s}f(u,u)(\tau)\,d\tau = e^{-t\nu A^s}u_0 + v[u,u](t).
\end{equation*}
%It is already known that $Q$ is well-defined from $B'$ to $B$ as long as $M$ is sufficiently small.
Thanks to Lemma \ref{lemma: estimate of nonlinear term in the critical case}, for all $u\in B$, $Qu-e^{-t\nu A^s}u_0 = v[u,u]\in B_{C_1M^2,T}^\beta$, where $C_1 = C_1(\alpha, \nu, \Omega, \beta)$.
Now requiring $M$ to be even smaller if necessary, such that $M\leq C_1^{-1}$, we obtain $Qu-e^{-t\nu A^s}u_0 \in B_{M,T}^\beta$.
This proves the claim.

To this end, with the smallness assumption on $M$, we define $u^{(0)} = e^{-t\nu A^s}u_0\in B$, and $u^{(j)} = Qu^{(j-1)}\in B$ for all $j\in \mathbb{N}_+$.
We shall prove by induction that for all $j\in \mathbb{N}_+$,
\begin{align}
&\;\|u^{(j)}-u^{(j-1)}\|_{L^\infty_T (D(A))\cap L^2_T (D(A^{1+s/2}))}\leq (C_2 M)^{j-1}C_0M^2,\label{eqn: geometric series in the original space}\\
&\;u^{(j)}-u^{(j-1)}\in B_{(C_3M)^{j+1}},\label{eqn: geometric series in the new space}
\end{align}
for some $C_2 = C(\alpha, \nu,\Omega)$ and $C_3 = C_3(\alpha,\nu,\Omega,\beta)$, while $C_0$ is defined in \eqref{eqn: estimate of the deviation from the semigroup homogeneous solution}.
Indeed, consider the equation for $\tilde{w}_j = u^{(j)}-u^{(j-1)}$.
For all $j\geq 2$,
\begin{equation*}
\partial_t\tilde{w}_{j} +\nu A^s \tilde{w}_{j} = f(\tilde{w}_{j-1}, u^{(j-1)})+f(u^{(j-2)}, \tilde{w}_{j-1}),\quad  \tilde{w}_{j}|_{t=0} = 0.
\end{equation*}
Again, by energy estimate, \eqref{eqn: norm of f used in the critical case} and \eqref{eqn: estimate for u in B},
\begin{equation}
\begin{split}
&\;\|\tilde{w}_j\|_{L^\infty_T(D(A))\cap L^2_T(D(A^{1+s/2}))}\\
\leq &\;C(\nu)(\|f(\tilde{w}_{j-1}, u^{(j-1)})\|_{L^2_T D(A^{1-s/2})}+\|f(u^{(j-2)}, \tilde{w}_{j-1})\|_{L^2_T D(A^{1-s/2})})\\
\leq &\;C_2M\|\tilde{w}_{j-1}\|_{L^\infty_T (D(A))\cap L^2_T (D(A^{1+s/2}))},
\end{split}
\label{eqn: estimate of the difference between two solutions}
\end{equation}
Now \eqref{eqn: geometric series in the original space} follows immediately from \eqref{eqn: estimate of the deviation from the semigroup homogeneous solution} and \eqref{eqn: estimate of the difference between two solutions}.
To show \eqref{eqn: geometric series in the new space}, we note that $u^{(1)}-u^{(0)} = v[u^{(0)},u^{(0)}]\in B_{C_5(C_4M)^2,T}^\beta$, where $C_4$ and $C_5$ are the constants in Lemma \ref{lemma: semigroup solution is in that class} and Lemma \ref{lemma: estimate of nonlinear term in the critical case}, respectively.
Assuming $4C_5>1$, we have
\begin{equation*}
u^{(1)}-u^{(0)} = v[u^{(0)},u^{(0)}]\in B_{(2C_5C_4M)^2,T}^\beta =: B_{(C_3M)^2,T}^\beta.
\end{equation*}
Now suppose $u^{(j)}-u^{(j-1)}\in B_{(C_3M)^{j+1},T}^\beta$, by Lemma \ref{lemma: semigroup solution is in that class} and Lemma \ref{lemma: estimate of nonlinear term in the critical case},
\begin{equation*}
\begin{split}
&\;u^{(j+1)}-u^{(j)}\\
= &\;v[u^{(j)}-u^{(j-1)}, u^{(j)}]+v[u^{(j-1)}, u^{(j)}-u^{(j-1)}]\\
\in&\; B_{2C_5(C_3M)^{j+1}C_4M,T}^\beta = B_{(C_3M)^{j+2},T}^\beta.
\end{split}
\end{equation*}

In the view of \eqref{eqn: geometric series in the original space} and \eqref{eqn: geometric series in the new space}, if $M$ is assumed to be even smaller if necessary such that $C_2M, C_3M<1$, then $\{u^{(j)}\}_{j\in\mathbb{N}}$ converges in $L^\infty_T (D(A))\cap L^2_T (D(A^{1+s/2}))$ and $C^\beta_{loc}((0,T]; D(A^{1+s/2}))$ to $u_{**}$.
It is easy to show that $u_{**}$ is a fixed-point of $Q$, and thus a local solution of \eqref{eqn: equivalent formulation of LANS equation}.
By uniqueness result in Proposition \ref{prop: local existence and uniqueness}, such $u_{**}$ is unique and $u_{**}=u_*$.
Here $u_*$ is the unique global solution from Theorem \ref{thm: global existence and uniqueness}.
It satisfies the following estimates
\begin{align*}
&\;\|u_*-e^{-t\nu A^s}u_0\|_{L^\infty_T (D(A))\cap L^2_T (D(A^{1+s/2}))}\leq \sum_{j = 1}^\infty (C_2 M)^{j-1}C_1M^2 =: C_6 M^2,\\
&\;u_*-e^{-t\nu A^s}u_0\in B^\beta_{\sum_{j = 1}^\infty(C_3M)^{j+1}, T} =:B^\beta_{C_7M^2, T}.
\end{align*}
Assuming $M\leq 1$, we obtain the desired estimates by estimates on the homogeneous semigroup solution.
%The uniqueness follows from Proposition \ref{prop: local existence and uniqueness}.

This completes the proof.
\end{proof}
\end{proposition}

Combining Proposition \ref{prop: local existence and uniqueness in critical case} with Theorem \ref{thm: global existence and uniqueness} yields the improved regularity of the global solution when $(n,s) = (2,1/2)$.

\begin{theorem}[Improved regularity of $u_*$ in the critical case]\label{thm: improved regularity in critical case}
Under the assumptions of Proposition \ref{prop: local existence and uniqueness in critical case}, the unique global solution $u_*\in C_{[0,+\infty)}(D(A))\cap L^2_{[0,+\infty),loc}(D(A^{1+s/2}))$ obtained in Theorem \ref{thm: global existence and uniqueness} satisfies %, solving \eqref{eqn: equivalent formulation of LANS equation} with initial condition $u|_{t = 0} = u_0$ (and boundary conditions \eqref{eqn: boundary conditions of LANS fractional} if $\Omega$ is a smooth bounded domain).
%Moreover, it satisfies \eqref{eqn: estimate of the global solution} and
\begin{enumerate}
%\item $\|u(t)\|_{D(A)}\leq C\|u_0\|_{D(A)}$ for all $t\in [0,+\infty)$;
\item $\|A^{s/2}u(t)\|_{D(A)}\leq C\|u_0\|_{D(A)}(1+t^{-1/2})$ for all $t\in (0,+\infty)$;
\item For all $0< t\leq t+h< +\infty$, $h\in[0,1]$,
\begin{align*}
\|u(t+h)-u(t)\|_{D(A)}\leq &\;Ch^\beta (1+t^{-\beta})\|u_0\|_{D(A)},\\
\|A^{s/2}(u(t+h)-u(t))\|_{D(A)}\leq &\;Ch^\beta (1+t^{-(\beta+1/2)})\|u_0\|_{D(A)},
\end{align*}
\end{enumerate}
where $C = C(\alpha,\nu,\Omega,\beta)$.
\end{theorem}

%For acknowledgements section, please don't number the section, please begin it with \section*{Acknowledgements}
\section*{Acknowledgments}
Zaihui Gan is partially supported by the National Science Foundation of China under grants
11571254 and the Program for New Century Excellent Talents in University (NCET-12-1058).
Fanghua Lin and Jiajun Tong are partially supported by National Science Foundation under
Award Number DMS-1501000. The research was initiated while the first author was visiting
the Courant Institute in the Fall of 2015.

% You may incorporate your references as follows in your main tex file.
% Using BibTex is not recommended but can be handled.

%\bibliographystyle{AIMS}
%\bibliography{On_the_Viscous_Camassa_Holm_Equation_with_Fractional_Diffusion}

\providecommand{\href}[2]{#2}
\providecommand{\arxiv}[1]{\href{http://arxiv.org/abs/#1}{arXiv:#1}}
\providecommand{\url}[1]{\texttt{#1}}
\providecommand{\urlprefix}{URL }

\medskip
%% The data information below will be filled by AIMS editorial staff
%Received xxxx 20xx; revised xxxx 20xx.
\medskip

\end{document}